\documentclass[12pt]{article}
\usepackage[english]{babel}
\usepackage{amsthm,amsfonts, amsbsy, amssymb,amsmath,graphicx}
\usepackage{graphics}

\newtheorem{theorem}{Theorem}
\newtheorem{lemma}{Lemma}
\newtheorem{corollary}{Corollary}

\newtheorem{proposition}{Proposition}
\theoremstyle{definition}
\newtheorem{definition}{Definition}
\theoremstyle{remark}
\newtheorem{remark}{Remark}
\newtheorem{example}{Example}

\def\Z{{\mathbb Z}}
\def\N{{\mathbb N}}

\newcommand{\V}{{\mathcal V}}

\newcommand{\eps}{\varepsilon}
\newcommand{\G}{{\mathcal G}}

\DeclareMathOperator{\id}{id}
\DeclareMathOperator*{\bigst}{\raisebox{-0.6ex}{\scalebox{2.5}{$\ast$}}}

\date{}

\title{Parity functors}
\author{Igor Nikonov}

\begin{document}

\maketitle

\begin{abstract}
A parity is a rule to assign labels to the crossings of knot diagrams in a way compatible with
Reidemeister moves. Parity functors can be viewed as parities which provide to each knot diagram its own
coefficient group that contains parities of the crossings. In the article we describe the universal oriented parity functors for free knots and for knots in a fixed surface.
\end{abstract}

\section*{Introduction}

V.O. Manturov~\cite{M3} defined a parity as a rule to assign labels $0$ and $1$
to the crossings of knot diagrams in a way compatible with
Reidemeister moves. Capability to discriminate between odd and even
crossings allows to treat the crossings differently when
transforming knot diagrams or calculating their invariants.
The notion of parity has proved to be an effective tools in knot
theory. It allows to strengthen knot invariants, to prove minimality
theorem and to construct (counter)examples~\cite{IM, IMN1, M1,M2,M3,M4,M5,M6,M7}.


In~\cite{IMN} parities with coefficients in an arbitrary abelian group were defined. In this paper we make one step to further generalization.
Parity functors can be described as parities which assign to each knot diagram its own
coefficient group that contains the parities of crossings of the diagram. The motivation for such a
generalization comes from the following example~\cite[Example 2.1]{IMN1}.




 \begin{figure}[h]
\centering\includegraphics[width=0.3\textwidth]{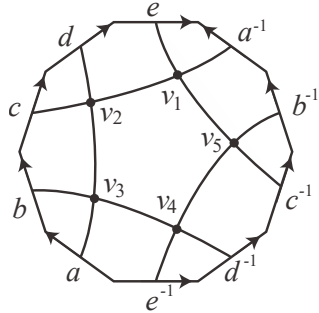} \caption{A flat knot
in a surface of genus two}\label{fig:flat_knot_example}
 \end{figure}

Let us consider the flat knot in Fig.~\ref{fig:flat_knot_example}, given as a curve in a surface of genus two, and take some parity $p$ with coefficients in an abelian group $A$. The knot has $\Z_5$-symmetry. By the symmetry the parities of all crossings
$p(v_i)$, $i=1,\dots,5$, must coincide. From the axioms of parity (cf.~\cite[Lemma 3.7]{IMN1}), the sum of parities of the vertices of the inner pentagon is zero: $$\sum_{i=1}^5 p(v_i)=0.$$ Then $5p(v_i)=0$ and $p(v_i)=0$. Thus, the parity of the crossings is trivial.

On the other hand, if the surface is fixed then one can assign nontrivial homological parity to the crossings~\cite[Example 2.1]{IMN1}. We can retain this nontriviality if we assume that a symmetry of the diagram induces an automorphism $\phi$ of the coefficient group $A$, so we will deal with equalities like $p(v_j)=\phi(p(v_i))$ instead of $p(v_j)=p(v_i)$.

The paper is organized as follows. In the next section we give the main definitions. We start with formal description of knot theories by means of diagrams and their transformations. Next we remind the notion of oriented parity given in~\cite{N} and then introduce the definition of oriented parity functor. We conclude the section with a construction of the universal oriented parity functor.

The second and the third sections are devoted correspondingly to description of the universal oriented parity functors for two knot theories: the free knots and knots in a given surface. The universal parity functor of free knots is a slight extension of the Gaussian parity. The universal functor for knots in a given surface coincides essentially with the homotopic parity. This means that there is no nontrivial parity functors for classical knots.

\section{Definitions}

A conventional form for presentation of knots and links are diagrams. Any knot admits infinitely many diagrams and any two of them can be linked with a sequence of elementary transformations --- diagram isotopies and Reidemeister moves.

Let us consider several examples of knots of various types and their diagrams.

A {\em  $4$-graph} is any union of four-valent graphs and {\em trivial components}, i.e. circles  considered as graphs without vertices and with one (closed) edge.
A {\em virtual diagram} is an embedding of a $4$-graph into plane so that each vertex of the graph is marked as either {\em classical} of {\em virtual} vertex.
At a classical vertex one a pair of opposite edges (called {\em overcrossing}) is chosen. The other pair of opposite edges at the vertex is called {\em undercrossing}. The vertices of the diagram are called also {\em crossings}.

Virtual crossings of a virtual diagram  are usually drawn circled. The undercrossing of a classical vertex is drawn with a broken line whereas the overcrossing is drawn with a solid line (see Fig.~\ref{fig:virtual_trefoil}). A diagram without virtual crossings is {\em classical}.
\begin{figure}[h]
\centering
\includegraphics[width=0.15\textwidth]{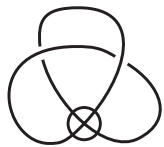}
\caption{Virtual trefoil diagram with two classical and one virtual crossings}\label{fig:virtual_trefoil}
\end{figure}
%

{\em Moves of virtual diagrams} include classical Reidemeister moves ($R1$, $R2$, $R3$) and {\em detour moves ($DM$)} that replace any diagram arc, which has only virtual crossings, with a new arc, which has the same ends and contains only virtual crossings (see Fig.~\ref{pic:virt_moves}). An equivalence class of virtual diagram modulo moves is called a {\em virtual link}~\cite{K1}.

\begin {figure}[h]
\centering
\includegraphics[width=0.3\textwidth]{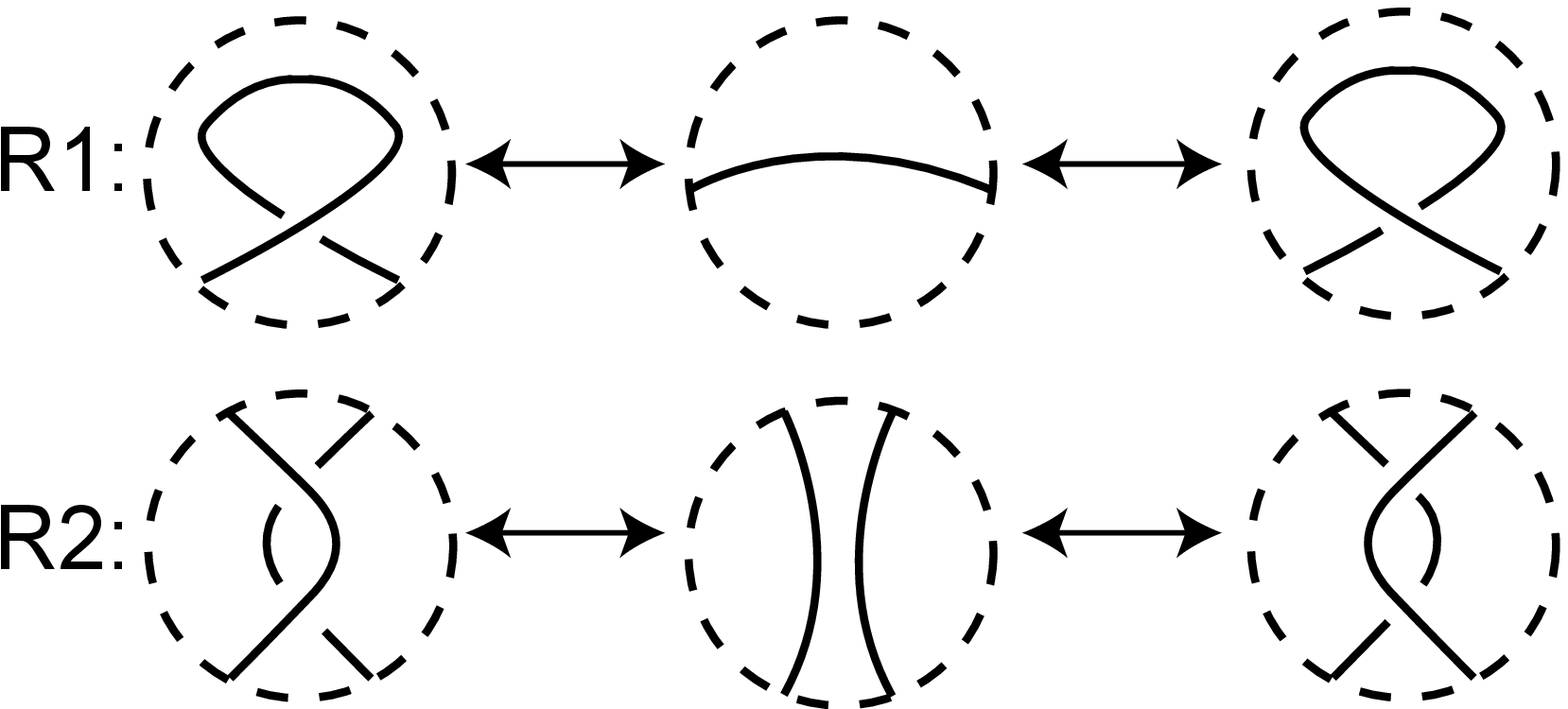}\qquad
\includegraphics[width=0.3\textwidth]{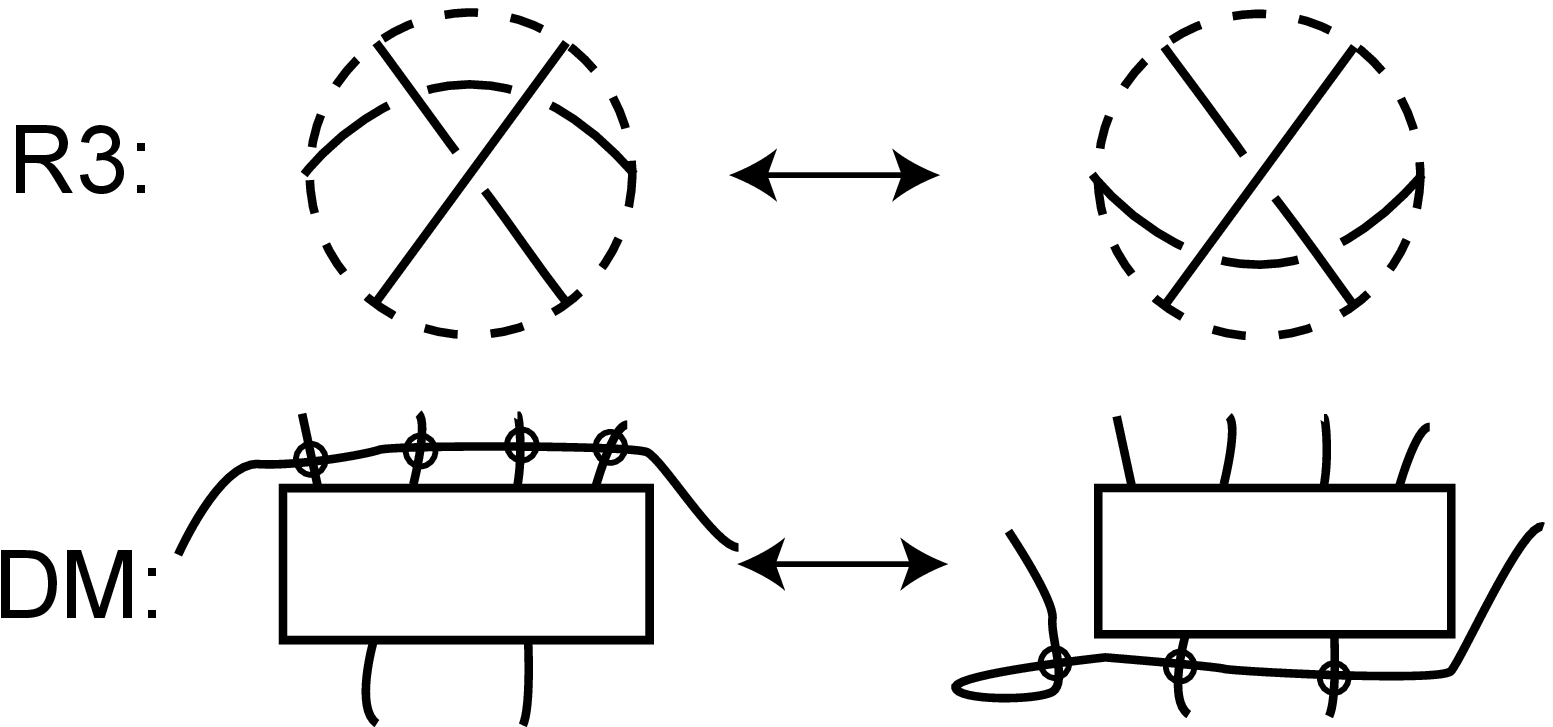}
\caption{Moves of virtual diagrams}\label{pic:virt_moves}
\end {figure}

A {\em classical link} is a link which has at least one classical diagram.

Kuperberg~\cite{Kuperberg} showed that any two classical diagrams that correspond to the same (virtual) link, can be connected by moves so that the intermediate diagrams are classical. By this reason, while working with classical knots and links, we can consider only classical diagrams and classical Reidemeister moves between them.

A {\em unicursal component} is a minimal set of diagram edges which is closed under passing from an edge to its opposite (at some end of the edge) edge. A diagram with one unicursal component is called a {\em diagram of a virtual knot}.

A {\em virtual knot} is an equivalence class of diagrams with one unicursal component.

There are other descriptions of virtual knots.

A {\em virtual link} can be viewed as an equivalence class of pairs $(S,D)$ where $S$ is a closed oriented surface and $D$ is a diagram in $S$ whose crossings are all classical~\cite{KK}. The equivalence relation is generated by diagram isotopies, classical Reidemeister moves and stabilizations (see Fig.~\ref{fig:stabilization}) which change the surface.

\begin{figure}[h]
\centering
  \includegraphics[width=0.5\textwidth]{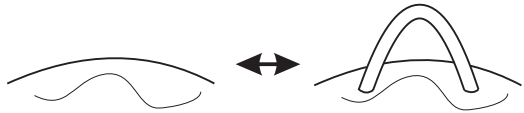}\\
  \caption{Stabilization}\label{fig:stabilization}
\end{figure}

If one excludes the stabilization moves she gets a knot theory in a given surface.

Let $S$ be a two-dimensional connected closed oriented surface. A {\em link (knot) in the surface $S$} is an equivalence class of  diagrams (diagrams with one unicursal component) in $S$  modulo diagram isotopies and classical Reidemeister moves $R1$, $R2$, $R3$.

Links and knots it the sphere $S^2$ are the classical links and knots.

Thirdly, (oriented) virtual knots can be defined by means Gauss diagrams~\cite{GPV}. The {\em Gauss diagram} $G=G(D)$ of a virtual knot diagram $D$ is a chord diagram whose chords correspond to the classical crossings of $D$, see Fig.~\ref{fig:virtual_gauss_diagram}. The chords carry an orientation (from over-crossing to under-crossing) and the sign of the crossings, see Fig.~\ref{fig:crossing_sign}.

\begin{figure}[h]
\centering\includegraphics[width=0.15\textwidth]{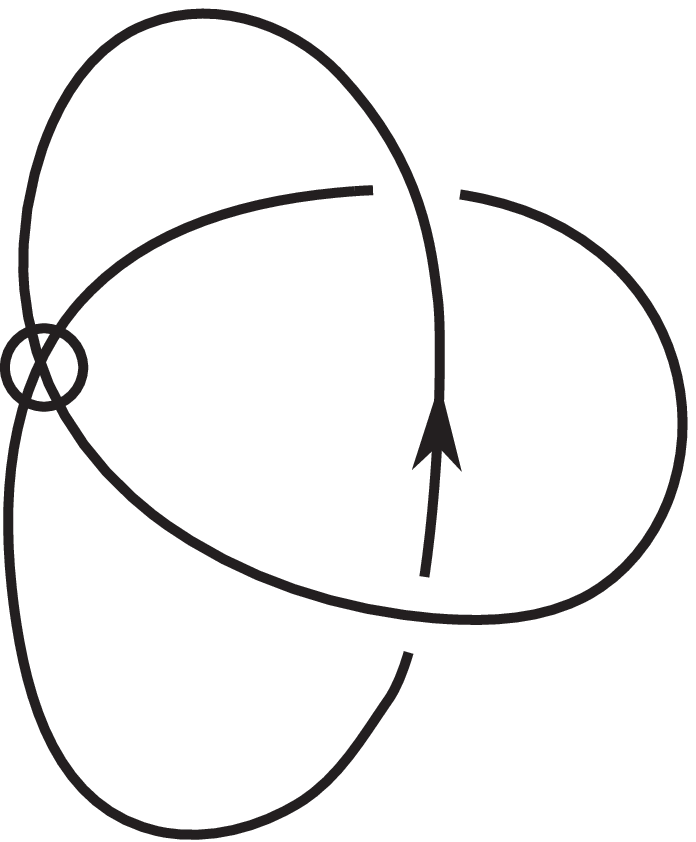}\qquad
\includegraphics[width=0.15\textwidth]{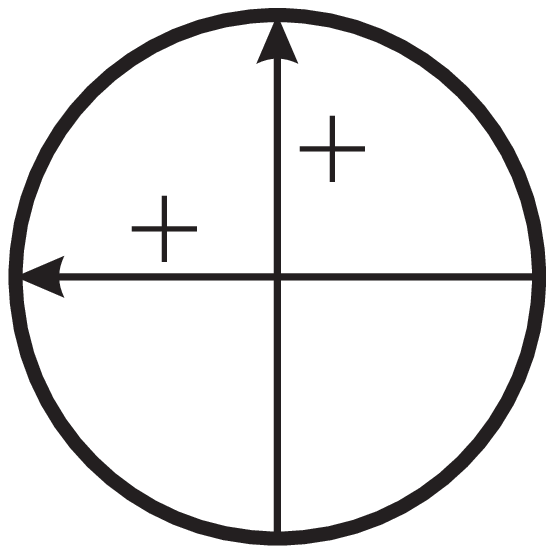}
\caption{Gauss diagram of the virtual trefoil}\label{fig:virtual_gauss_diagram}
\end{figure}

\begin{figure}[h]
\centering\includegraphics[width=0.2\textwidth]{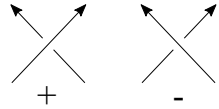}
\caption{The sign of a crossing}\label{fig:crossing_sign}
\end{figure}

Classical Reidemeister moves induce transformations of Gauss diagrams (see Fig.~\ref{fig:reidemeister_gauss}). Virtual knots are exactly the equivalence classes of Gauss diagrams modulo the induced Reidemeister moves.

\begin{figure}[h]
\centering\includegraphics[width=0.6\textwidth]{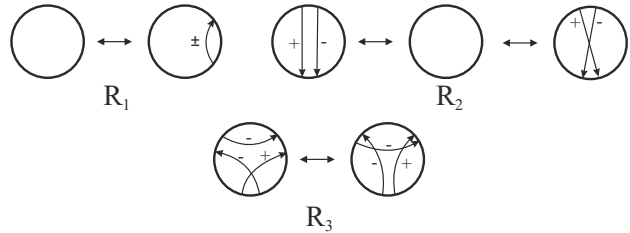}
\caption{Reidemeister moves on Gauss diagrams}\label{fig:reidemeister_gauss}
\end{figure}

The three given definitions of virtual knots are equivalent~\cite{CKS,GPV}.

If one admits crossing switch transformations (Fig.~\ref{fig:crossing_switch}) of virtual diagrams, i.e. neglects the over-undercrossing structure, one gets the theory of {\em flat knots}. For knots in a fixed surface, factorization by crossing switch gives the theory of {\em generally immersed curves in the surface}.

\begin{figure}[h]
\centering\includegraphics[width=0.25\textwidth]{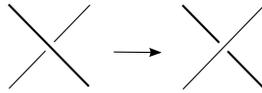}
\caption{Crossing switch}\label{fig:crossing_switch}
\end{figure}

The further factorization by virtualization move (Fig.~\ref{fig:virtualization}) leads to the theory of {\em free knots}.

\begin{figure}[h]
\centering\includegraphics[width=0.25\textwidth]{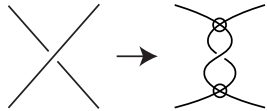}
\caption{Virtualization move}\label{fig:virtualization}

\end{figure}

Crossing switch and virtualization moves induces transformations on Gauss diagrams (Fig.~\ref{fig:gauss_virtualization}). Therefore, the theory of free knots which is obtained by factorization by these tranformations, is the theory of chord diagrams without orientation and labels on the chords. The moves of free knots are the moves in Fig.~\ref{fig:reidemeister_gauss} after one has wiped the arrowheads and signs.

\begin{figure}[h]
\centering\includegraphics[width=0.25\textwidth]{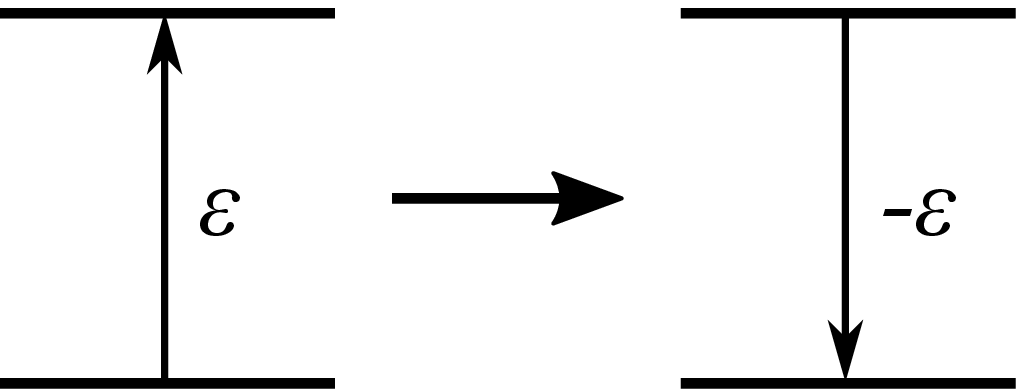}
\qquad\qquad\includegraphics[width=0.25\textwidth]{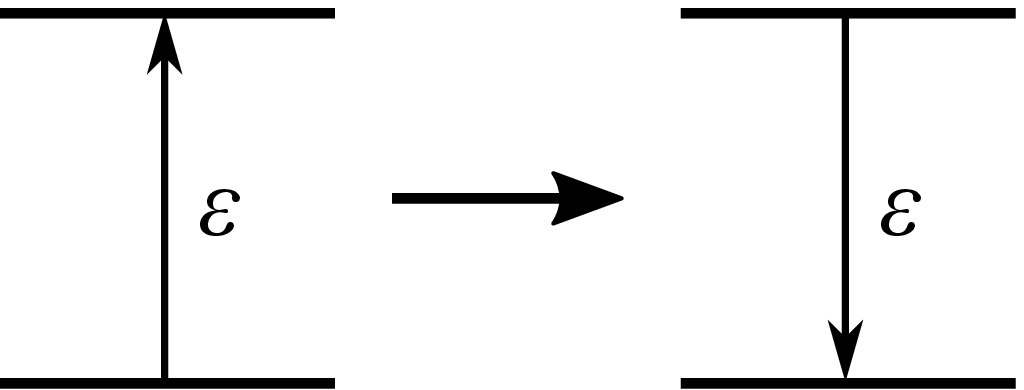}
\caption{Crossing switch (left) and virtualization (right) of a Gauss diagram}\label{fig:gauss_virtualization}
\end{figure}

%
%
%


\subsection {Category of knot diagrams}

Let $\mathcal{K}$ be a knot. We shall use the notion of `knot' in
one of the following situations:
 \begin{itemize}
  \item virtual knot;
  \item flat knot;
  \item free knot;
  \item knot in a given surface
  \item homotopy class of curves immersed in a given surface;
  \item classical knot.
 \end{itemize}

\begin{definition}\label{def:knot_category}
The {\em category of diagrams $\mathfrak{K}$ of the knot $\mathcal{K}$}
is the small category whose objects are the diagrams of $\mathcal{K}$ and
morphisms are (formal) compositions
of {\em elementary morphisms}. By an elementary morphism we mean
 \begin{itemize}
  \item an isotopy of a diagram;
  \item an isomorphism between diagrams (in case of virtual, flat or free knot);
  \item a Reidemeister move.
 \end{itemize}
\end{definition}

Since the number of vertices of a diagram may change under
Reidemeister moves, there is no bijection between the sets of
vertices of two diagrams connected by a sequence of Reidemeister
moves. To construct any connection between two sets of vertices we
introduce the notion of a partial bijection which means just
the bijection between the subsets of vertices corresponding to each
other in the two diagrams.

 \begin{definition}\label{def:partial_bijection}
 Let $X$ and $Y$ be two sets. A {\em partial bijection} between $X$ and $Y$ is a triple
$(\widetilde X,\widetilde Y,\phi)$, where $\widetilde X\subset X$,
$\widetilde Y\subset Y$ and $\phi\colon\widetilde X\to \widetilde Y$
is a bijection. The subset $dom(\phi)=\widetilde X$ is the {\em domain}, and $im(\phi)=\widetilde Y$ is the {\em image} of the partial bijection.

Let $G$ and $H$ be two groups. A {\em partial isomorphism} between $G$ and $H$ is a triple
$(\widetilde G,\widetilde H,\phi)$, where $\widetilde G\subset G$,
$\widetilde H\subset H$  are subgroups and $\phi\colon\widetilde G\to \widetilde H$
is an isomorphism.
 \end{definition}

Sets with partial bijections form a category $\mathfrak{S}_{part}$, and groups with partial isomorphisms form a category $\mathfrak G_{part}$. Let $\mathfrak A_{part}$ be the subcategory of $\mathfrak G_{part}$ which consists of the abelian groups and partial isomorphisms between them.

\begin{definition}\label{def:vertex_functor}
  Let $\mathfrak{K}$ be a knot and $\mathfrak{K}$ be its diagram category. The {\em vertex functor} is a
functor $\V$ from $\mathfrak{K}$ to the $\mathfrak{S}_{part}$ such that for each diagram
$K\in Ob(\mathfrak K)$ we define $\V(K)$ to be the set of crossings of $K$ (or chords of $K$ if it is a chord diagram). Any elementary morphism
$f\colon K\to K'$ naturally induces a partial bijection $\V(f)\colon\V(K)\to\V(K')$. For a morphism $f=f_1\circ\cdots\circ f_n$, where $f_i$, $i=1,\dots,n$ are elementary, the partial bijection $\V(f)$ is defined as the composition $\V(f_1)\circ\cdots\circ\V(f_n)$.
\end{definition}

Below we will use the notation $f_*(v)$ for $\V(f)(v)$, $v\in dom(\V(f))$, for short.

\begin{remark}
  For an elementary morphism $f\colon K\to K'$ the partial map $\V(f)$ is a bijection in all cases except a first or a second Reidemeister move. If $f$ is an increasing (resp., decreasing) first Reidemeister move then the bijection domain includes $\V(K)$ and differs from $\V(K')$ by one element (resp., includes $\V(K')$ and differs from $\V(K)$ by one element). If $f$ is a second Reidemeister move the partial bijection $\V(f)$ embraces all the crossings of the diagrams $K$ and $K'$ except two appearing or disappearing crossings.
\end{remark}

\subsection{Parities and parity functors}

Let $A$ be an abelian group. Let us recall the definition of parity~\cite{IMN}.

 \begin{definition}\label {def:parity}
A {\em parity $p$ on diagrams of a knot $\mathcal{K}$ with
coefficients in $A$} is a family of maps $p_D\colon\V(D)\to A$,
$D\in Ob(\mathfrak{K})$, such that for any elementary morphism
$f\colon D\to D'$ the following holds:
  \begin{itemize}
   \item[(P0)]
$p_{D'}(f_*(v))=p_DK(v)$ for any $v\in dom(f_*)$;
   \item[(P1)]
$p_D(v)=0$ if $f$ is the decreasing first Reidemeister move and $v$
is the disappearing crossing;
   \item[(P2)]
$p_D(v_1)+p_D(v_2)=0$ if $f$ is a decreasing second Reidemeister
move and $v_1,\,v_2$ are the disappearing crossings;
   \item[(P3)]
$p_D(v_1)+p_D(v_2)+p_D(v_3)=0$ if $f$ is a third Reidemeister move
and $v_1,\,v_2,\,v_3$ are the crossings participating in this move.
  \end{itemize}
\end{definition}

Among examples of parities are the Gaussian parity, the link parity (with coefficients in $\Z_2$) of virtual knots and the homological parity of knots in a given surface (with coefficients in the first homology group of the surface)~\cite{IMN1}.

\begin{example}[Gaussian parity]
Let $D$ be a diagram of a knot $\mathcal K$ and $c\in\V(D)$ be a crossing of $D$. The crossing $c$ splits $D$ into two halves (see Fig.~\ref{fig:knot_halves}). The Gaussian parity $gp_D(c)\in\Z_2$ of $c$ is the parity of the number of crossing points on any half.

\begin{figure}[h]
\centering\includegraphics[width=0.5\textwidth]{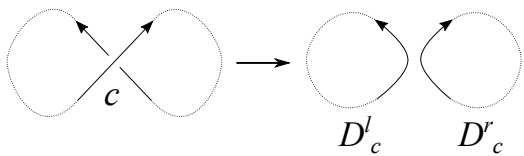}
\caption{Halves of the diagram}\label{fig:knot_halves}
\end{figure}
\end{example}

\begin{example}[Link parity]\label{exa:link_parity}
Let $D$ be a diagram of some link $\mathcal L$ with two components. The crossings of the diagram $D$ can be divided into {\em self-crossings} of a component of $\mathcal L$ and {\em mixed crossings} where two different components intersect. One assigns the parity $0$ to the self-crossings and the parity $1$ to the mixed crossings.
\end{example}

\begin{example}[Homological parity]
Let $\mathcal K$ be a knot in a fixed oriented surface $S$ and $D\subset S$ be its diagram. Any crossing $v\in\V(D)$ splits the diagram into left and right halves. This halves are cycles in the surface. The {\em homological parity} of the crossing $v$ is the element $p^h_D(v)=[D^l_v]\in H_1(S,\Z_2)/[\mathcal K]$ where $[\mathcal K]$ is the homology class of the knot $\mathcal K$.
\end{example}

 \begin{definition}\label{def:parity_functor}
A {\em parity functor $P$ on the diagrams of a knot $\mathcal{K}$} is a pair $(P, \mathcal A)$ where $\mathcal A$ is a functor from the category of its diagrams $\mathfrak K$ to the category $\mathfrak A_{part}$, and $P$ is a natural map between the vertex functor $\V$ and the coefficient functor $\mathcal A$ (in fact, the composition of $\mathcal A$ with the forgetting functor $\mathfrak A_{part}\to\mathfrak S_{part}$), i.e. a
set of maps $P_D\colon\V(D)\to \mathcal A(D)$ for
each diagram $D$, such that for any elementary morphism $f\colon
D\to D'$:
  \begin{itemize}
\item[(P0)] for any crossing $v\in dom(f_*)\subset \V(D)$ one has $P_D(v)\in dom(\mathcal A(f))$ and $P_{D'}(f_*(v))=\mathcal A(f)(P_D(v))$;
\item[(P1)] $P_D(v)=0$ if $f$ is a decreasing first Reidemeister move and $v$
is the disappearing crossing;
\item[(P2)] $P_D(v_1)+P_D(v_2)=0$ if $f$ is a decreasing second Reidemeister
move and $v_1,\,v_2$ are the disappearing crossings;
\item[(P3)] $P_D(v_1)+P_D(v_2)+P_D(v_3)=0$ if $f$ is a third Reidemeister move
and $v_1,\,v_2,\,v_3$ are the crossings participating in the move.
  \end{itemize}
 \end{definition}

\begin{remark}
1. The first condition of the definition above represents the fact $P$ is a natural map.

2. Any parity $p$ with coefficients in a group $A$ is a parity functor with the constant coefficient functor $\mathcal A\equiv A$ (i.e. $\mathcal A(D)=A$ for any diagram $D$  and $\mathcal A(f)=\id_A$ for any morphism $f\colon D\to D'$ ), and maps $P_D=p_D\colon \V(D)\to A$, $D\in Ob(\mathfrak K)$.
\end{remark}

Many properties of parities with coefficients can be extended to parity functors. For example, the following statement shows that for parity functors, parities of crossings are defined modulo 2 (cf.~\cite[Lemma 4.4]{IMN}).

 \begin{proposition}\label{prop:parity_mod2}
Let $(P,\mathcal A)$ be a parity functor. Then $2P_D(a)=0$ for any diagram $D\in Ob(\mathfrak K)$ and any crossing $a\in\V(D)$.
 \end{proposition}

 \begin{proof}
Apply a second Reidemeister move $f$ to the diagram $D$ as shown in Fig.~\ref{fig:parity_mod2}. Then $a_1=f_*(a)$. Since one can apply a decreasing second Reidemeister move to the crossings $a_1$ and $b_1$, we have $P_{D_1}(a_1)+P_{D_1}(b_1)=0$. Apply two second Reidemeister moves and one third Reidemeister move to $D_1$ and get the diagram $D_2$. Then $a_2=g_*(a_1)$ and $b_2=g_*(b_1)$. Then
$$P_{D_2}(a_2)+P_{D_2}(b_2)=\mathcal A(g)(P_{D_1}(a_1))+\mathcal A(g)(P_{D_1}(b_1))=\mathcal A(g)(0)=0.$$

One can apply third Reidemeister moves to the triples $a_2, c_2, d_2$ and $b_2,c_2,d_2$. Hence
$P_{D_2}(a_2)+P_{D_2}(c_2)+P_{D_2}(d_2)=0$ and
$P_{D_2}(b_2)+P_{D_2}(c_2)+P_{D_2}(d_2)=0$. Hence, $P_{D_2}(a_2)=P_{D_2}(b_2)$
and $2P_{D_2}(a_2)=0$. Since $\mathcal A(f)$ and $\mathcal A(g)$ are partial isomorphisms we have $2P_{D_1}(a_1)=0$ and $2P_{D}(a)=0$.

 \begin{figure}
\centering\includegraphics[width=0.7\textwidth]{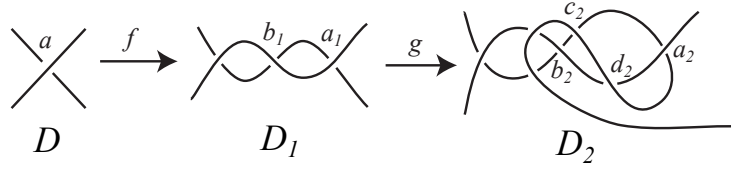} \caption{Proof of $2P_D(a)=0$} \label{fig:parity_mod2}
 \end{figure}
 \end{proof}

In~\cite{N} the definition of parity was modified in order to evade the relation proved above.

\begin{definition}\label{def:oriented_parity}
Let $G$ be a group (may be noncommutative). Assume that a knot $\mathcal K$ is oriented, so on the diagrams of the diagram category $\mathfrak K$ the induced orientation is given.

An {\em oriented parity} $p$ on the diagram category $\mathfrak K$ is a family of maps $p_D\colon \V(D)\to G$, $D\in Ob(\mathfrak K)$, that possesses the following properties:
\begin{itemize}
\item[(P0)] $p_{D}(v)=p_{D'}(f_*(v))$ for any elementary morphism $f\colon D\to D'$ and any crossing $v\in\V(D)$ in the domain of the partial bijection $f_*$;

\item[(P1)] if $f\colon D\to D'$ is a decreasing first Reidemeister move and $v\in\V(D)$ is the disappearing crossing then $p_D(v)=1$ where $1$ is the unit of the coefficient group $G$;

\item[(P2)] $p_D(v_1)p_D(v_2)=1$ when $f$ is a decreasing second Reidemeister move and $v_1,\,v_2$ are the disappearing crossings;

\item[(P3+)]
 if $f\colon D\to D'$ is a third Reidemeister move then
 $$p_D(v_1)^{\epsilon(v_1)}p_D(v_2)^{\epsilon(v_2)}p_D(v_3)^{\epsilon(v_3)}=1$$
  where $v_1,v_2,v_3$  are the vertices involved in the move and $\epsilon(v_i)$ is the incidence index of the vertex $v_i$ to the disappearing triangle, see Fig.~\ref{pic:incidence_index} left, and the order of the vertices is induced by the orientation of the surface the diagram lies in.

\begin {figure}[h]
\centering
\includegraphics[width=0.4\textwidth]{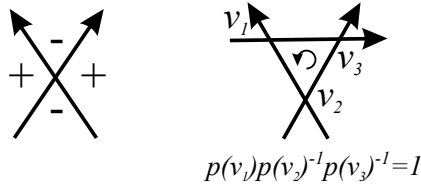}
\caption{Incidence index and relations of oriented parity}\label{pic:incidence_index}
\end {figure}
\end{itemize}
\end{definition}

In Fig.~\ref{pic:incidence_index} right one can see an example of a relation for an oriented parity. The motivation example for the property (P3+) is the homotopical parity (see Definition~\ref{def:homotopical_parity}).

\begin{remark}\label{rem:oriented_parity}
1. The properties (P1),(P2) and (P3+) can be unified in the following property: $\prod p_D(v_i)^{\epsilon(v_i)}=1$ if $f$ is a decreasing first or second Reidemeister move or a third Reidemeister move, $v_i$ are the crossings which take part in the move and
$\epsilon(v_i)$ is the incidence index of the vertex to the disappearing region.

2. Oriented parity does not depend on the orientation of the knot. More precise, if a family of maps $p$ is an oriented parity on the diagram category of the knot $\mathcal K$ then it is an oriented parity on the diagram category of the knot $-\mathcal K$ which is obtained from $\mathcal K$ by orientation reversion (we identify naturally the diagram categories of the knots $\mathcal K$ and $-\mathcal K$).

3. Oriented parities include parities with coefficients defined above because of equalities $2p(v)=0$.

4. For free knots, there is no canonical orientation of the triangle in a third Reidemeister move. Therefore, we take commutative coefficient groups when consider oriented parities on free knot.
\end{remark}

An example of oriented parity (with coefficients in $\Z$) is the index of crossings.

\begin{example}[Index parity]
 Let $D$ be a diagram of an oriented virtual knot $\mathcal K$ and $v\in\V(D)$ be a crossing of $D$. Let us calculate the crossing points on the left half $D^l_v$ of the diagram at the crossing $v$ with the signs as in Fig.~\ref{fig:classical_crossing_sign}. The sum is the {\em index parity} $ip_D(v)$ of the crossing $v$. Note that $ip_D(v)$ coincides with the intersection index of~\cite{Henrich} and differs by the sign of the crossing $v$ from the index $W_K(v)$ from~\cite{K2}: $ip_D(v)=sgn(v)\cdot W_K(v)$ (see also~\cite{Cheng}).

  \begin{figure}[h]
\centering\includegraphics[width=0.35\textwidth]{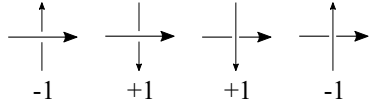}
\caption{Intersection number of a classical crossing}\label{fig:classical_crossing_sign}
\end{figure}

Since the index parity does not use under-overcrossing structure, this parity is defined for flat knots.
\end{example}

\begin{definition}\label{def:oriented_parity_functor}
An {\em oriented parity functor $P$ on the diagrams of an oriented knot $\mathcal{K}$} is a pair $(P, \mathcal G)$ where $\mathcal G$ is a functor from the category of its diagrams $\mathfrak K$ to the category $\mathfrak G_{part}$, and $P$ is a natural map between the vertex functor $\V$ and the coefficient functor $\mathcal G$, i.e. a set of maps $P_D\colon\V(D)\to \mathcal G(D)$ for each diagram $D$, such that for any elementary morphism $f\colon D\to D'$:

\begin{itemize}
\item[(P0)] for any crossing $v\in dom(f_*)\subset \V(D)$ one has $P_D(v)\in dom(\mathcal G(f))$ and $P_{D'}(f_*(v))=\mathcal G(f)(P_D(v))$;

\item[(P1)] $P_D(v)=1$ if $f$ is a decreasing first Reidemeister move and $v$ is the disappearing crossing;

\item[(P2)] $P_D(v_1)P_D(v_2)=1$ if $f$ is a decreasing second Reidemeister move and $v_1,\,v_2$ are the disappearing crossings;

\item[(P3+)] if $f\colon D\to D'$ is a third Reidemeister move then
$$P_D(v_1)^{\epsilon(v_1)}P_D(v_2)^{\epsilon(v_2)}P_D(v_3)^{\epsilon(v_3)}=1$$
  where $v_1,v_2,v_3$  are the vertices involved in the move and ordered counterclockwise, and $\epsilon(v_i)$ is the incidence index of the vertex $v_i$ in relation with the disappearing triangle.
\end{itemize}

 \end{definition}

 An oriented parity functor $(P,\mathcal G)$ is called {\em trivial} if for any $D\in Ob(\mathfrak K)$ and any $v\in\V(D)$ $P_D(v)=1$.

 Let us consider basic properties of oriented parity functors.

 \begin{proposition}
 Let $(P,\mathcal G)$ be an oriented parity functor on diagrams of a knot $\mathcal K$ and $D\in Ob(\mathfrak K)$ be a diagram of $\mathcal K$.

 1. Assume that crossings $v_1,v_2\in\V(D)$ form a bigon (Fig.~\ref{pic:bigon_trigon} left). Then $P_D(v_1)^{\epsilon(v_1)}P_D(v_2)^{\epsilon(v_2)}=1$ where $\epsilon(v_i)$ are the incidence indices to the bigon;
 2. Assume that crossings $v_1,v_2,v_3\in\V(D)$ form a triangle (Fig.~\ref{pic:bigon_trigon} right). Then $P_D(v_1)^{\epsilon(v_1)}P_D(v_2)^{\epsilon(v_2)}P_D(v_3)^{\epsilon(v_3)}=1$.

\begin {figure}[h]
\centering
\includegraphics[width=0.2\textwidth]{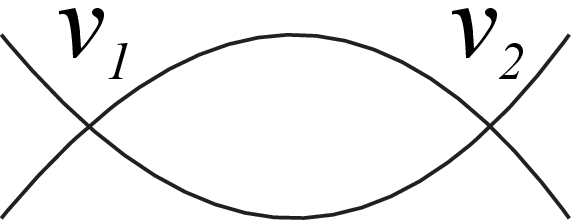}\qquad
\includegraphics[width=0.15\textwidth]{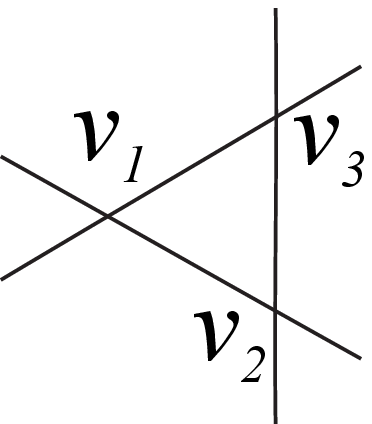}
\caption{A bigon and a triangle. The under-overcrossing structure can be any.}\label{pic:bigon_trigon}
\end {figure}
 \end{proposition}

\begin{proof}
1. Let $v_1,v_2$ be the vertices of a bigon in the diagram $D$. Then $\epsilon(v_1)=\epsilon(v_2)$, hence, the equation $$P_D(v_1)^{\epsilon(v_1)}P_D(v_2)^{\epsilon(v_2)}=1$$ is equivalent to $P_D(v_1)P_D(v_2)=1$. If a second Reidemeister move can be applied to the bigon, the equality holds due to the property(P2).

If the bigon is alternating then apply a first Reidemeister move $f$ as shown in Fig.~\ref{fig:alt_bigon}. With some abuse of notation, we will write $v_i$ for $f_*(v_i)\in\V(D')$, $i=1,2$. By the properties  (P1) and (P3+) we have
$$P_{D'}(v_1)^{\epsilon(v_1)}P_{D'}(v_2)^{\epsilon(v_2)}P_{D'}(w)^{\epsilon(w)}=1$$
 and $P_{D'}(w)=1$. Hence, $P_{D'}(v_1)^{\epsilon(v_1)}P_{D'}(v_2)^{\epsilon(v_2)}=1$. The left part is the image of $P_D(v_1)^{\epsilon(v_1)}P_D(v_2)^{\epsilon(v_2)}$ by the partial isomorphism $\mathcal G(f)$. Thus, $P_D(v_1)^{\epsilon(v_1)}P_D(v_2)^{\epsilon(v_2)}=1$.

\begin {figure}[h]
\centering
\includegraphics[width=0.4\textwidth]{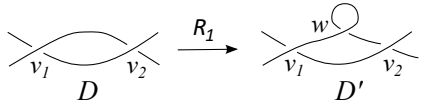}
\caption{Proof for an alternating bigon}\label{fig:alt_bigon}
\end {figure}

2. Let $v_1,v_2,v_3$ be the vertices of a triangle in the diagram $D$. If a third Reidemeister move can be applied to the triangle, the equality follows from the property (P3+).

\begin {figure}[h]
\centering
\includegraphics[width=0.5\textwidth]{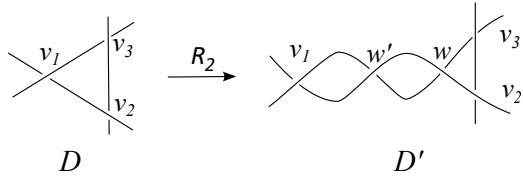}
\caption{Proof for an alternating triangle}\label{fig:alt_trigon}
\end {figure}

If $v_1,v_2,v_3$ form an alternating triangle, apply a second Reidemeister move $f\colon D\to D'$ as shown in Fig.~\ref{fig:alt_trigon}. Then $\epsilon(v_1)=\epsilon(w')=\epsilon(w)$. By the property (P2) and the previous result, $P_{D'}(v_1)P_{D'}(w')=P_{D'}(w)P_{D'}(w')=1$, hence $P_{D'}(v_1)=P_{D'}(w)$. By the property (P3) the equality, we have
$$P_{D'}(w)^{\epsilon(w)}P_{D'}(v_2)^{\epsilon(v_2)}P_{D'}(v_3)^{\epsilon(v_3)}=1.$$
Thus, $P_{D'}(v_1)^{\epsilon(v_1)}P_{D'}(v_2)^{\epsilon(v_2)}P_{D'}(v_3)^{\epsilon(v_3)}=1$ and $$P_D(v_1)^{\epsilon(v_1)}P_D(v_2)^{\epsilon(v_2)}P_D(v_3)^{\epsilon(v_3)}=1.$$
\end{proof}

\begin{definition}
Let $(P,\mathcal G)$ and $(P',\mathcal G')$ be oriented parity functors on the diagram category $\mathfrak K$ of some knot $\mathcal K$. A {\em homomorphism} between the parity functors is a family of homomorphisms $\Phi_D\colon \mathcal G(D)\to\mathcal G'(D)$ such that
\begin{enumerate}
  \item for any diagram $D\in Ob(\mathfrak K)$ $P'_D=\Phi_D\circ P_D$;
  \item for any elementary morphism $f\colon D\to D'$ one has the equality $\Phi_{D'}\circ\mathcal G(f)=\mathcal G'(f)\circ\Phi_D$ of homomorphisms from $dom(\mathcal G(f))$ to $\mathcal G'(D')$.
\end{enumerate}
\end{definition}

\begin{definition}
Let $(P,\mathcal G)$  be an oriented parity functor on the diagram category $\mathfrak K$ of some knot $\mathcal K$. We can asssociate  an oriented parity with the parity functor. To be more precise, let
$$
G^P=\bigst_{D\in Ob(\mathfrak K)} \mathcal G(D) / \langle x^{-1}\cdot\mathcal G(f)(x)\, |\, f\in Mor(\mathfrak K), x\in dom(\mathcal G(f))\rangle
$$
be the direct limit of coefficient groups of the parity functor, and $p^P_D\colon\V(D)\to G^P$ be the composition of $P_D$ and the natural homomorphism $\iota_D\colon \mathcal G(D)\to G^P$. Then the family $p^P_D$ is an oriented parity with coefficients in the group $G^P$, called the {\em associated parity to the parity functor} $(P,\mathcal G)$.
\end{definition}

 The associated parity possesses the following universal property.

\begin{proposition}\label{prop:universal_associated_parity}
Let $(P,\mathcal G)$  be an oriented parity functor on the diagram category $\mathfrak K$ of some knot $\mathcal K$, and $p^P$ be the oriented parity associated to $(P,\mathcal G)$. Then for any homomorphism $\Phi$ from the parity functor  $(P,\mathcal G)$ to an oriented parity $p$ with coefficients in some group $G$ there exists a unique group homomorphism $\phi\colon G^P\to G$ such that $\Phi_D=\phi\circ\iota_D$ for any $D\in Ob(\mathfrak K)$. In other word, any homomorphism from a parity functor $(P,\mathcal G)$ to a parity $p$  with coefficients passes through a homomorphism from the associated parity $p^P$ to $p$.
\end{proposition}

Note that if $p$ is an oriented parity on a diagram category, considered as a parity functor with a constant coefficient group, then the associated parity coincides with the original parity: $p^p=p$.

\subsection{Universal parity functor}

\begin{definition}
An oriented parity functor $(P^u, \mathcal G^u)$ on the diagrams of a knot $\mathcal K$ is called a {\em universal parity functor} if for any parity functor $(P,\mathcal G)$ on the diagrams of $\mathcal K$ there exists a unique family of homomorphisms
$$\phi_D\colon \mathcal G^u(D)\to\mathcal G(D),\ D\in Ob(\mathfrak K),$$
 such that $P_D=\phi_D\circ P^u_D$ for any $D\in Ob(\mathfrak K)$.
\end{definition}

\begin{theorem}\label{thm:universal_parity_functor}
Let $\mathcal K$ be a knot and $\mathfrak K$ be the category of its diagrams. Then there exists a unique (up to isomorphism) universal parity functor $(P^u, \mathcal G^u)$ on $\mathfrak K$.
\end{theorem}

\begin{proof}
The uniqueness follows from the universal property of the universal parity functor. Let us give a construction of such a functor.

For any $D\in\mathfrak K$ denote $F(D)=\langle\V(D)\rangle$ the free group generated by the set of crossings $\V(D)$ of the diagram $D$. For any $f\colon D\to D'$ define a partial isomophism $\mathcal F(f)$ from $F(D)$ to $F(D')$ with the domain $\langle dom(f_*)\rangle\subset F(D)$ and the codomain $\langle im(f_*)\rangle\subset F(D')$ by the formula $\mathcal F(f)(v)=f_*(v)$, $v\in dom(f_*)\subset \V(D)$.

Consider a sequence of increasing normal subgroups of $F(D)$ defined by induction. Let $H_0(D)\lhd F(D)$ be the normal subgroup generated by the relations (P1), (P2) and (P3+). Given the subgroups $H_i(D)$, we define the subgroup $H_{i+1}(D)$ as the normal closure of the set
$$H_i(D)\cup \bigcup_{f\colon D\to D'\mbox{ \scriptsize is elementary}} (\mathcal F(f))^{-1}(H_i(D'))$$ in $F(D)$. Finally, let $H_\infty(D)=\bigcup_{i=1}^\infty H_i(D)$ and $\mathcal G^u(D)=F(D)/H_\infty(D)$. The map $P^u_D$ is the composition $\V(D)\hookrightarrow F(D)\to\mathcal G^u(D)$. For any elementary morphism $f\colon D\to D'$ the partial isomorphism of $\mathcal G^u(f)$ is induced by the map $\mathcal F(f)$.

Let us check the properties of a universal oriented parity functor. The properties (P1),(P2) and (P3+) hold because $H_0\subset H_\infty=\ker P^u_D$.

By definition $\bigcup_{f\colon D\to D'} (\mathcal F(f))^{-1}(H_\infty(D'))\subset H_\infty(D)$. In particular, if $f\colon D\to D'$ is an elementary morphism then
$$(\mathcal F(f))^{-1}(H_\infty(D'))=(\mathcal F(f))^{-1}(H_\infty(D')\cap im(\mathcal F(f)))\subset H_\infty(D).$$ Analogously,
$$(\mathcal F(f^{-1}))^{-1}(H_\infty(D))=\mathcal F(f)(H_\infty(D)\cap dom(\mathcal F(f)))\subset H_\infty(D').$$ Then $\mathcal F(f)$ establishes an isomorphisms of groups
$H_\infty(D)\cap dom(\mathcal F(f))$ and $H_\infty(D')\cap im(\mathcal F(f))$. Hence, the induced map $\mathcal G^u(f)$ is a well defined partial isomorphism with the domain
$$dom(\mathcal G^u(f))=dom(\mathcal F(f))/(H_\infty(D)\cap dom(\mathcal F(f)))$$
and the codomain
$$im(\mathcal G^u(f))=im(\mathcal F(f))/(H_\infty(D')\cap im(\mathcal F(f))).$$
The property (P0) follows from the definition of $\mathcal F(f)$.

Let us check the universal property. Let $(P,\mathcal G)$ be an oriented parity functor. For any $D\in Ob(\mathfrak K)$ the map $P_D\colon\V(D)\to \mathcal G(D)$ induces a homomorphism $\tilde\phi_D\colon F(D)\to \mathcal G(D)$. Let $\Pi_D=\ker\tilde\phi_D\subset F(D)$ be the kernel of $\tilde\phi_D$. We need to show that $\tilde\phi_D$ induces a correct homomorphism $\phi_D\colon \mathcal G^u(D)\to \mathcal G(D)$, i.e. that $H_\infty(D)\subset \Pi(D)$.

Since the parity functor $(P,\mathcal G)$ obeys the properties (P1),(P2),(P3+), the kernel $\Pi(D)$ contains $H_0(D)$.

Let $f\colon D\to D'$ be an elementary morphism. The property (P0) implies that $\tilde\phi_D(dom(\mathcal F(f)))\subset dom(\mathcal G(f))$ and $\tilde\phi_{D'}(im(\mathcal F(f)))\subset im(\mathcal G(f))$, and $\mathcal G(f)\circ\tilde\phi_D=\tilde\phi_{D'}\circ\mathcal F(f)$ because $\mathcal G(f)$ is a partial isomorphism. Then $(\mathcal F(f))^{-1}(\Pi(D'))\subset\Pi(D)$. Hence,

$$\bigcup_{f\colon D\to D'\mbox{ \scriptsize is elementary}} (\mathcal F(f))^{-1}(\Pi(D'))\subset \Pi(D).$$

Assume that $H_i(D)\subset \Pi(D)$ for all diagrams $D\in Ob(\mathfrak K)$. Then
$$\bigcup_{f\colon D\to D'} (\mathcal F(f))^{-1}(H_i(D'))\subset \bigcup_{f\colon D\to D'} (\mathcal F(f))^{-1}(\Pi(D'))\subset \Pi(D).$$
Therefore, $H_{i+1}(D)\subset \Pi(D)$.
Thus, $H_i(D)\subset\Pi(D)$ for all $i$, and $H_\infty(D)\subset\Pi(D)$.

\end{proof}

\begin{remark}
The oriented parity $p^{P^u}$ associated with the universal parity functor $(P^u,\mathcal G^u)$ have the following universal property: for any oriented parity $p$ with coefficients in some group $G$ there exists a unique homomorphism $\phi\colon \mathcal (G^u)^{P^u}\to G$ such that for any diagram $D\in Ob(\mathfrak K)$ the parity map $p_D$ is the composition $\phi\circ p^{P^u}_D$. The associated parity $p^u=p^{P^u}$ with coefficients in $G^u=(G^u)^{P^u}$ is called the {\em universal oriented parity} on the diagram category $\mathfrak K$.
\end{remark}

\begin{definition}
An oriented parity functor $(P, \mathcal G)$ is {\em reduced} if for any diagram $D\in Ob(\mathfrak K)$ the set $\{P_D(v)\,|\, v\in\V(D)\}$ generates the coefficient group $\mathcal G(D)$.

Given an oriented parity functor $(P, \mathcal G)$, we define its {\em reduction} as the pair $(\bar P,\bar{\mathcal G})$ where for any $D\in Ob(\mathfrak K)$
$\bar{\mathcal G}(D)=\langle P_D(v)\,|\, v\in\V(D)\rangle\subset \mathcal G(D)$ and $\bar P_D$ is the corestriction of $P_D$ to the map from $\V(D)$ to $\bar{\mathcal G}(D)$.

The {\em canonical reduction} $(\tilde P, \tilde{\mathcal G})$ of an oriented parity functor $(P, \mathcal G)$ is defined as follows. For any diagram $D\in Ob(\mathfrak K)$, let $\langle\V(D)\rangle$ be the free group generated by the set $\V(D)$ and $P_D\colon \langle\V(D)\rangle\to \mathcal G(D)$ be the induced group homomorphism. We set $\tilde{\mathcal G}(D)=\langle\V(D)\rangle/\ker P_D$ to be the coimage of $P_D$, and $\tilde P_D\colon\V(D)\to\tilde{\mathcal G}(D)$ to be the composition $\V(D)\hookrightarrow\langle\V(D)\rangle\to\tilde{\mathcal G}(D)$.

Note that the reduction $(\bar P,\bar{\mathcal G})$ and the canonical reduction $(\tilde P, \tilde{\mathcal G})$ are isomorphic parity functors.
\end{definition}

Reduction means we exclude all inessential parts of the coefficient groups of the parity functor.

\begin{proposition}\label{prop:associated_reduction_parity}
Let $p$ be an oriented parity on a diagram category $\mathfrak K$ with coefficients in a group $G$. Let $(P,\mathcal G)$ be the reduction of the parity $p$, and $p'$ be the associated parity to $(P,\mathcal G)$. Then the coefficient group of $p'$ can be identified with the subgroup $G'\subset G$ generated by values $p_D(v), D\in Ob(\mathfrak K), v\in\V(D)$, and $p'$ is the corestriction of $p$ to $G'$.
\end{proposition}
\begin{proof}
For any diagram $D\in Ob(\mathfrak K)$ the coefficient group $\mathcal G(D)$ is the subgroup $\langle p_D(v)\,|\, v\in\V(D)\rangle\subset G$, and the partial isomorphisms  $\mathcal G(f)$ are the identity maps on the domain. Hence, the associated coefficient group is the subgroup of $G$ generated by the subgroups $\mathcal G(D)$, i.e. the subgroup $G'$.
\end{proof}

\begin{proposition}
Universal parity functor in any knot theory is reduced.
\end{proposition}

\begin{proof}
The statement follows from the construction of the universal parity functor in Theorem~\ref{thm:universal_parity_functor}.
\end{proof}


\section{Oriented Gaussian parity functor}

In this section we construct the universal oriented parity functor for the diagrams of a free knot. Below we consider free knot diagrams as chord diagrams.

Let $\mathcal K$ be a free knot and $D$ be its diagram, i.e. a chord diagram. We denote crossings (chords) of the diagram with small letters $a,b,\dots$  and the ends of the chord with letters $a',a'',b',b'',\dots$. A {\em polygon} in a diagram $D$ is a sequence $a_1'a_1''a_2'a_2''\dots a_k'a_k''$ of chord ends such that $a_i'$ and $a_i''$, $i=1,\dots,k$, are the ends of one chord $a_i$, and the points $a_i''$ and $a_{i+1}'$ are adjacent points on the diagram circle (i.e. there is no chord end inside the segment $a_i''a_{i+1}'$).

Let $(P,\G)$ be an oriented parity functor on the diagram category $\mathfrak K$ of the free knot $\mathcal K$. As it was mentioned in Remark~\ref{rem:oriented_parity}, we assume that coefficient groups $\G(D)$ are commutative.

\begin{lemma}\label{lem:free_parity_polygon}
For any polygon $\pi=a_1'a_1''a_2'a_2''\dots a_k'a_k''$ in a free
knot diagram $D$ and any oriented parity functor $P$ one has the
identity
$$
\sum_{i=1}^k\eps_\pi(a_i)P_D(a_i)=0
$$
where $\eps_\pi(a_i)=1$ if the segments $a_{i-1}''a_i'$ and
$a_i''a_{i+1}'$ have the same orientation, and $\eps_\pi(a_i)=-1$ if
the segments have opposite orientations.
\end{lemma}

 \begin{figure}
\centering\includegraphics[width=0.75\textwidth]{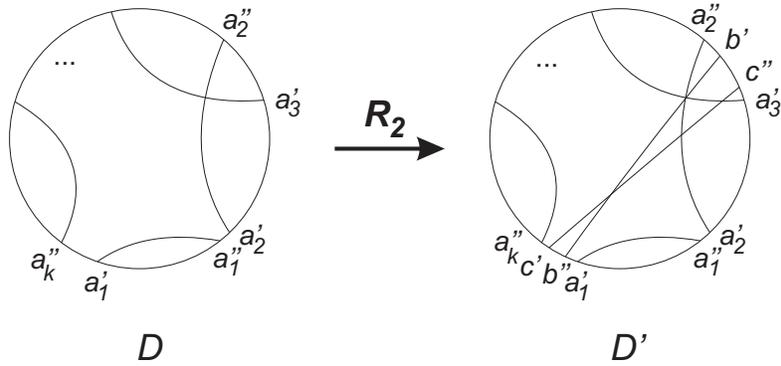}
\caption{Polygon identity: induction step}
\label{fig:free_parity_polygon}
 \end{figure}

\begin{proof}
We use induction on $k$. The cases $k=1,2,3$ are equivalent to the
parity properties (P1), (P2) and (P3+).

Let $k>3$. Add points $c'$ and $b''$ in the segment $a_k''a_1'$ and
add points $b'$ and $c''$ in the segment $a_2''a_3'$. Apply a second
Reidemeister move to the diagram $D$ by adding the cord $b$ with the
ends $b', b''$ and the cord $c$ with the endpoints $c', c''$. Denote
the new diagram $D'$ (see Fig.~\ref{fig:free_parity_polygon}). The
polygon $\pi$ splits into the triangle $\pi'=a_1'a_1''a_2'a_2''b'b''$
and the polygon $\pi''=c'c''a_3'a_3''\dots a_k'a_k''$. By induction we
have
$\eps_{\pi''}(c)P_{D'}(c)+\sum_{i=3}^k\eps_{\pi''}(a_i)P_{D'}(a_i)=0$
and
$\eps_{\pi'}(a_1)P_{D'}(a_1)+\eps_{\pi'}(a_2)P_{D'}(a_2)+\eps_{\pi'}(b)P_{D'}(b)=0$.
Since $\eps_{\pi'}(a_i)=\eps_{\pi}(a_i)$, $i=1,2$,
$\eps_{\pi''}(a_i)=\eps_{\pi}(a_i)$, $i=3,\dots, k$ and
$\eps_{\pi'}(b)=\eps_{\pi''}(c)$ and $P_{D'}(b)+P_{D'}(c)=0$, then
$\sum_{i=1}^k\eps_\pi(a_i)P_{D'}(a_i)=0.$ Hence
$\sum_{i=1}^k\eps_\pi(a_i)P_{D}(a_i)=0.$
\end{proof}

\begin{lemma}\label{lem:free_parity_4P}
Let $a$ be a chord in a free knot diagram $D$. Then for any oriented
parity functor $P$ we have $4P_D(a)=0$.
\end{lemma}

\begin{figure}[h]
\centering\includegraphics[width=0.6\textwidth]{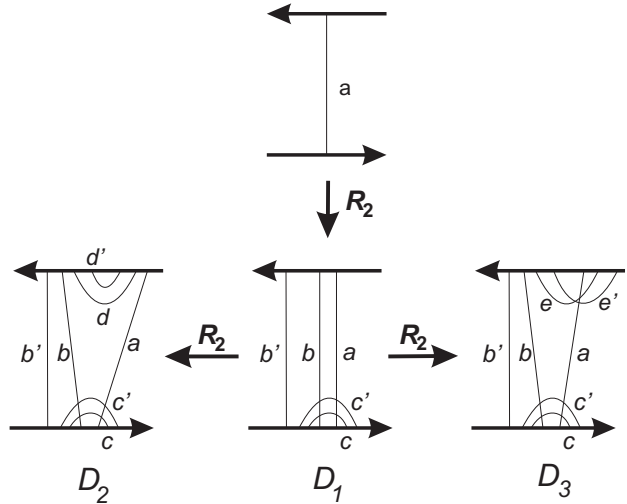}
\caption{Proof of the equality
$4P_D(a)=0$}\label{fig:free_parity_4P}
\end{figure}

\begin{proof}
By applying second Reidemeister moves we add chords $b,b'$, $c,c'$
(diagram $D_1$ in Fig.~\ref{fig:free_parity_4P}). We add further
pairs of chords $d,d'$ and $e,e'$ to obtain diagrams $D_2$ and $D_3$.
Sequences of chords $bad$ and $bcad$ form polygons in the diagram
$D_2$. Then from Lemma~\ref{lem:free_parity_polygon} we have
$P_{D_2}(b)+P_{D_2}(a)+P_{D_2}(d)=0$ and
$-P_{D_2}(b)+P_{D_2}(c)-P_{D_2}(a)+P_{D_2}(d)=0$. Then
$P_{D_2}(c)=2P_{D_2}(a)+2P_{D_2}(b)$ that implies the equality
$P_{D_1}(c)=2P_{D_1}(a)+2P_{D_1}(b)$ in the coefficient group
$P(D_1)$.

On the other hand, the polygons $bae$ and $bcae$ in $D_3$ yield
$P_{D_3}(b)-P_{D_3}(a)-P_{D_3}(e)=0$ and
$-P_{D_3}(b)+P_{D_3}(c)+P_{D_3}(a)-P_{D_3}(e)=0$. Hence
$P_{D_3}(c)=-2P_{D_3}(a)+2P_{D_3}(b)$ and
$P_{D_1}(c)=-2P_{D_1}(a)+2P_{D_1}(b)$.

Comparing the formulas for $P_{D_1}(c)$ we get the equality
$4P_{D_1}(a)=0$, so $4P_{D}(a)=0$.
\end{proof}

\begin{lemma}\label{lem:free_parity_even}
Let $a$ be an Gaussian even chord and $b_1, b_2,\dots, b_{2k}$ are the ends
of chords on one of the halves of the diagram $D$ at $a$. Then for
any oriented parity functor $P$ we have the equality
$$
P_D(a)= 2\sum_{i=1}^{2k} P_D(b_i).
$$
\end{lemma}

\begin{proof}
We prove the lemma by induction on $k$. The case $k=0$ follows from
the property (P1).

Let $k=1$. If $b_1$, $b_2$ are the ends of one chord then it can be
removed with a first Reidemeister move. After that the cord $a$ can
also be removed with a first Reidemeister move. Then by property
(P1)  we have $P_D(b_1)=P_D(b_2)=0$ and $P_D(a)=0$ so
$P_D(a)=2(P_D(b_1)+P_D(b_2))$.

\begin{figure}[h]
\centering\includegraphics[width=0.5\textwidth]{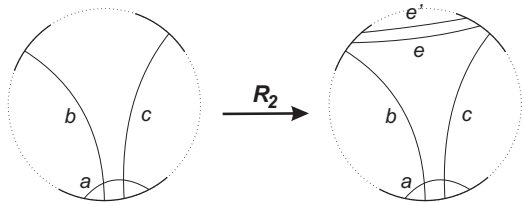}
\caption{Case $k=1$}\label{fig:free_parity_even_two}
\end{figure}

If the ends $b_1$ and $b_2$ belong to different chords then their
other ends lie in the other half of the diagram. Add with a second
Reidemeister move chords $c$ and $c'$ (see
Fig.~\ref{fig:free_parity_even_two}). Then chords $b_1, b_2, c$ and
chord $b_1,a,b_2,c$ form polygons. By
Lemma~\ref{lem:free_parity_polygon} we have
$P_{D'}(b_1)+P_{D'}(b_2)+P_{D'}(c)=0$ and
$-P_{D'}(b_1)+P_{D'}(a)-P_{D'}(b_2)+P_{D'}(c)=0$. Hence
$P_{D'}(a)=2(P_{D'}(b_1)+P_{D'}(b_2))$, so
$P_{D}(a)=2(P_{D}(b_1)+P_{D}(b_2))$.

\begin{figure}[h]
\centering\includegraphics[width=0.75\textwidth]{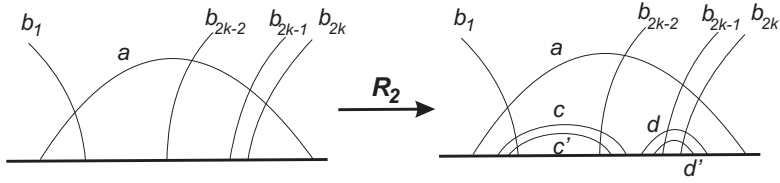}
\caption{Induction step}\label{fig:free_parity_even_step}
\end{figure}

Assume that the formula is valid for even chords with $2k$ ends on a
half. Let $a$ be a chord with $2k+2$ ends on a half at $a$. By two
second Reidemeister moves we add chords $c, c'$ and $d,d'$ as shown
in Fig.~\ref{fig:free_parity_even_step}. Then by induction
$P_D(c')=2\sum_{i=1}^{2k}P_D(b_i)$ and
$P_D(d')=2(P_D(b_{2k+1})+P_D(b_{2k+2}))$. From properties (P2) and
(P3) we have $P_D(c)+P_D(c')=0$, $P_D(d)+P_D(d')=0$ and
$P_D(a)+P_D(c)+P_D(d)=0$. Hence
$$
P_D(a)=-P_D(c)-P_D(d)=P_D(c')+P_D(d')=2\sum_{i=1}^{2k+2}P_D(b_i).
$$
\end{proof}

\begin{corollary}\label{crl:free_parity_even}
If a chord $a$ of the diagram $D$ is even for Gaussian parity then for
any oriented parity
\begin{enumerate}
\item $2P_D(a)=0$;
\item $P_D(a)= 2\sum_{\textstyle b \colon b\mbox{ is Gaussian odd and linked with }a} P_D(b)$.
\end{enumerate}
\end{corollary}
\begin{proof}
The first statement follows immediately from
Lemmas~\ref{lem:free_parity_4P},\ref{lem:free_parity_even}. The
second statement follows from Lemma~\ref{lem:free_parity_even} and
the first statement.
\end{proof}

\begin{lemma}\label{lem:free_parity_difference}
Let $a$ and $b$ be chords of a free knot diagram $D$ that have the
same Gaussian parity. Let $c_1,\dots,c_k$ and $d_1,\dots,d_l$ are
the ends of chords in two opposite segments separating the chords
$a$ and $b$. Then
$$
P_D(a)= (-1)^{k+1}P_D(b)+2\left(\sum_{i=1}^k P_D(c_i)+\sum_{j=1}^l
P_D(d_j)\right).
$$
\end{lemma}
\begin{proof}
Assume that chords $a$ and $b$ are not linked and $k$ is even (so
$l$ is even too). Then add chords $e,e'$ and $f,f'$ with second
Reidemeister moves (see Fig.~\ref{fig:free_parity_even_even}). Then
in the new diagram $D'$ we have the equalities
$P_{D'}(a)+P_{D'}(e)+P_{D'}(b)+P_{D'}(f)=0$,
$P_{D'}(e)+P_{D'}(e')=0$, $P_{D'}(f)+P_{D'}(f')=0$ and
$P_{D'}(e')=2\sum_{i=1}^k P_{D'}(c_i)$, $P_{D'}(f')=2\sum_{j=1}^l
P_{D'}(d_j)$ due to Lemma~\ref{lem:free_parity_even}. Then
$$
P_{D'}(a)=-P_{D'}(b)-P_{D'}(e)-P_{D'}(f)=-P_{D'}(b)+2\sum_{i=1}^k
P_{D'}(c_i)+2\sum_{j=1}^l P_{D'}(d_j),
$$
so $P_{D}(a)=-P_D(b)+2\left(\sum_{i=1}^k P_D(c_i)+\sum_{j=1}^l
P_D(d_j)\right)$.

\begin{figure}[h]
\centering\includegraphics[width=0.75\textwidth]{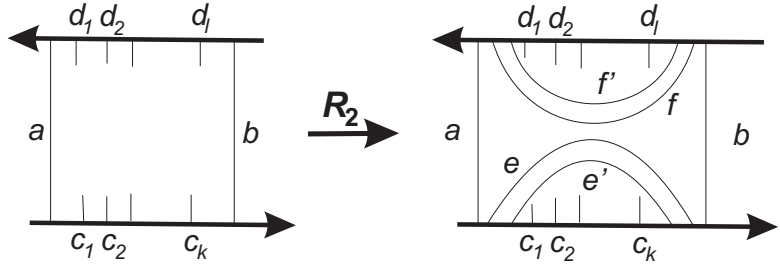}
\caption{Case $k$ and $l$ are even}\label{fig:free_parity_even_even}
\end{figure}

Suppose that chords $a$ and $b$ are not linked and $k$ and $l$ are odd. Add with second Reidemeister moves auxiliary chords $e,e',f,f',g,g'$. In the new diagram we have polygons $ad_1gc_1$, $bfd_1gc_1e$, $ee'$, $ff'$ that give the equalities $P_{D'}(a)-P_{D'}(d_1)+P_{D'}(g)-P_{D'}(c_1)=0$,
$P_{D'}(b)+P_{D'}(f)+P_{D'}(d_1)+P_{D'}(g)+P_{D'}(c_1)+P_{D'}(e)=0$,
$P_{D'}(e)+P_{D'}(e')=0$, $P_{D'}(f)+P_{D'}(f')=0$. From Lemma~\ref{lem:free_parity_even} we have also $P_{D'}(e')=2\sum_{i=2}^k P_{D'}(c_i)$, $P_{D'}(f')=2\sum_{j=2}^l
P_{D'}(d_j)$. Then
\begin{multline*}
P_{D'}(a)=P_{D'}(b)+P_{D'}(e)+P_{D'}(f)+2P_{D'}(c_1)+2P_{D'}(d_1)=\\
P_{D'}(b)+2\sum_{i=1}^k P_{D'}(c_i)+2\sum_{j=1}^l P_{D'}(d_j).
\end{multline*}
Here we used the equalities $P_{D'}(e)=-P_{D'}(e)$, $P_{D'}(f)=-P_{D'}(f)$ that follows from Corollary~\ref{crl:free_parity_even}.

Thus, $P_{D'}(a)=P_D(b)+2\left(\sum_{i=1}^k P_D(c_i)+\sum_{j=1}^l
P_D(d_j)\right)$.

\begin{figure}[h]
\centering\includegraphics[width=0.8\textwidth]{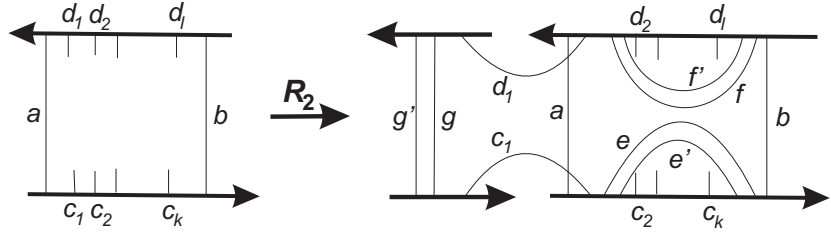}
\caption{Case $k$ and $l$ are odd}\label{fig:free_parity_odd_odd}
\end{figure}

The case when the chords $a$ and $b$ are linked can be treated analogously.

\end{proof}

The following statement is an immediate consequence of Lemmas~\ref{lem:free_parity_4P} and \ref{lem:free_parity_difference}.
\begin{corollary}
If $a$ and $b$ are Gaussian odd chords in a free knot diagram $D$
then for any oriented parity $P$ we have $2P_D(a)=2P_D(b)$.
\end{corollary}

We denote the element $2P_D(a)\in P(D)$, $a$ is a Gaussian odd chord
in the diagram $D$, by $z_D\in P(D)$. Then $2z_D=0$.

\begin{proposition}\label{prop:free_parity_odd_even}
Let $P$ be an oriented parity functor on free knots, $D$ be a free knot diagram.
\begin{itemize}
\item If $a$ is a Gaussian even chord in $D$ then
$$
P_D(a)=
\left\{
  \begin{array}{cl}
    0, & \hbox{if $a$ is linked with even number of Gaussian odd chords;} \\
    z_D, & \hbox{if $a$ is linked with odd number of Gaussian odd chords.}
  \end{array}
\right.
$$
\item Let $a$ and $b$ be Gaussian odd chords in $D$ and $k_e$ and $l_e$ (resp. $k_o$ and $l_o$) be the numbers of ends of even (resp. odd) chords on the two opposite segments separating the chords $a$ and $b$. Then
$$
P_D(a)=P_D(b)+(k_e+l_o+1)z_D.
$$
\end{itemize}
\end{proposition}

\begin{proof}
1) The first statement follows from Corollary~\ref{crl:free_parity_even} and the definition of $z_D$.

2) From Lemma~\ref{lem:free_parity_difference}, Corollary~\ref{crl:free_parity_even} and the definition of $z_D$ we have
\begin{multline*}
P_D(a)=(-1)^{k_e+k_o+1}P_D(b)+(k_e\cdot 0+k_oz_D+l_e\cdot 0+l_oz_D)=\\ (-1)^{k_e+k_o+1}P_D(b)+(k_o+l_o)z_D= P_D(z)+(k_e+k_o+1)z_D+(k_o+l_o)z_D=\\ P_D(b)+(k_e+l_o+1)z_D
\end{multline*}
where the equality $(-1)^{k_e+k_o+1}P_D(b)=P_D(b)+(k_e+k_o+1)z_D$ follows from $P_D(b)=-P_D(b)+z_D$.
\end{proof}

\begin{remark}
We can divide the set of chords $\V(D)$ of a free knot diagram $D$ into four classes $E_0, E_1, O', O''$ where
\begin{itemize}
\item $E_0$ consists of Gaussian even chords which are linked with even number of Gaussian odd chords;
\item $E_1$ consists of Gaussian even chords which are linked with odd number of Gaussian odd chords;
\item $O'$ and $O''$ form a splitting of the set of Gaussian odd chords such that two odd chords $a$ and $b$ belong to the same class $O'$ or $O''$ iff the corresponding number $k_e+l_o+1$ (see Proposition~\ref{prop:free_parity_odd_even}) is even.
\end{itemize}

For any oriented parity functor $P$ parities of any two chord from one class $E_0, E_1, O'$ or  $O''$ coincide according to Proposition~\ref{prop:free_parity_odd_even}.

Note that classes $E_1, O',O''$ are empty if the diagram has no odd chords.
\end{remark}

\begin{definition}\label{def:oriented_Gaussian_parity}
We define the {\em oriented Gaussian parity functor} $(P^{og}, \G^{og})$ as follows: for any free knot diagram $D$  we set $\G^{og}(D)=0$ to be trivial group if $D$ has no Gaussian odd chords, and $\G^{og}(D)\simeq\Z_4$ if $D$ has Gaussian odd chords. In the latter case we identify $\G^{og}(D)$ with the four-element set $\{ E_0,E_1,O',O''\}$ of classes of crossings. Isomorphism of $\G^{og}(D)$ with $\Z_4$ is defined so that $E_0$ corresponds to $0$, $E_1$ corresponds to $2$ and the classes $O',O''$ correspond to $1$ and $3$. The map $P^{og}_D\colon \V(D)\to\G^{og}(D)$ is the natural projection.
\end{definition}

\begin{theorem}\label{thm:free_parity_universal}
Oriented Gaussian parity functor is correctly defined and is the universal oriented parity functor on free knots.
\end{theorem}
\begin{proof}
Let us check that the groups $\G^{og}(D)$ and the maps $P^{og}_D$ obey the conditions (P0)-(P3) of parity functor definition. Let $f\colon D_1\to D_2$ be a Reidemeister move.

Property (P0). Let $a$ be a Gaussian even chord in $D_1$ that does not disappear after the move $f$. We should show that the parity of $a$ does not change. If $a$ does not take part in the move then the number of ends of Gaussian odd chords on a half of the diagram at $a$ remains unchanged or changes by $2$ or $4$ (in the case of second Reidemeister move on two odd chords). The other possibility is $f$ is a third Reidemeister move on chords $a,b,c$. Then the chords $b$ and $c$ are either both Gaussian even or both Gaussian odd. In the latter case the number of odd ends on a half of the diagram at $a$ can change by $2$, in other cases the number does not change. Therefore, if $a$ belongs to class $E_0$ (or $E_1$) then the corresponding chord in $E_2$ belongs to the same class $E_0$ (or $E_1$).

Let $a$ and $b$ be Gaussian odd chords in $D_1$ that preserve after move $f$. Assume that $a$ and $b$  belong to the same class $O'$ (or $O''$). This means that the number $k_e+l_o$ is odd (see Proposistion~\ref{prop:free_parity_odd_even}). We should show that the corresponding chords $a',b'$ in $D_2$ belong to the same class. There are several cases.

Let $a$ and $b$ don't take part in the move. Then the number $k_e$ after the move does not change or changes by $2$ (when $f$ is a first Reidemeister move or a second move on two even chords) or $4$ (when $f$ is a second move on two even chords). Analogously, the number $l_o$ can change only by $2$ or $4$ (when $f$ is a second move on two odd chords). Thus, the move $f$ does not change the parity of the sum $k_e+l_o$, so the chords $a',b'$ in $D_2$ belong to the same class.

Let $f$ be a third Reidemeister move on chords $a,c,d$. Then one of the chords $c,d$ is Gaussian even and the other is Gaussian odd. In this case the numbers $k_e$, $l_o$ either remain the same or change by $1$, so the sum $k_e+l_o$ change by an even number. Hence, the chords $a',b'$ in $D_2$ belong to the same class.

Let $f$ be a third Reidemeister move on chords $a,b,c$. Then the chord $c$ is Gaussian even. In this case the numbers $k_e$ either remain the same or change by $2$ (if the both ends of $c$ belong to the segment between $a$ and $b$), the number $l_o$ does not change. So the sum $k_e+l_o$ change by an even number and the chords $a',b'$ in $D_2$ belong to the same class.

Analogous reasonings show that if $a$ and $b$ be Gaussian odd chords in $D_1$ that preserve after move $f$ and belong to different classes then corresponding chords in $D_2$ lie in different classes too. Thus, the property (P0) is valid for the oriented Gaussian parity.

If $\G^{og}(D)=0$ then properties (P1), (P2), (P3) are trivial. So we suppose further that $D$ contains Gaussian odd chords and $\G^{og}(D)\simeq\Z_4$.

Property (P1). If $f$ is a decreasing first Reidemeister move that remove chord $a$ then $a$ belongs to the class $E_0$ since there is a half at $a$ that contains no odd ends. So $P^{og}_{D_1}(a)=0$.

Property (P2). Let $f$ be a decreasing Reidemeister move that remove chords $a$ and $b$. If the $a$ and $b$ are even then their halves contain the same number of odd ends, so both $a$ and $b$ belong to $E_0$ or $E_1$. In the first case we have $P^{og}_{D_1}(a)+P^{og}_{D_1}(b)=0+0=0$, in the second case $P^{og}_{D_1}(a)+P^{og}_{D_1}(b)=2+2=0\in\Z_4$.

Assume now that the chords $a$ and $b$ are odd. We can take segments between $a$ and $b$ so that $k_e=l_o=0$. Then the chords belong to different odd classes and we have $P^{og}_{D_1}(a)+P^{og}_{D_1}(b)=1+3=0\in\Z_4$.

Property (P3). Let $f$ be a third Reidemeister move on chords $a,b,c$. Then either the chords $a,b,c$ are all Gaussian even or there are one even and two odd chords among them. If $a,b,c$ are Gaussian even we can take halves at these chords so that the halves cover the circle of the diagram and there is no ends of chords in the intersection of any two halves. Let $A,B,C$ are the number of ends of Gaussian odd chords in the halves at $a,b$ and $c$ correspondingly. Then the sum $A+B+C$ is the number of ends of odd chords in the diagram, so the sum is even. Since $P^{og}_{D_1}(a)=2A, P^{og}_{D_1}(b)=2B, P^{og}_{D_1}(c)=2C$ due to Proposition~\ref{prop:free_parity_odd_even}, we have
$$
P^{og}_{D_1}(a)+P^{og}_{D_1}(b)+P^{og}_{D_1}(c)=2(A+B+C)=0\in\Z_4.
$$
Note that $P^{og}_{D_1}(d)=-P^{og}_{D_1}(d)$ for any Gaussian even chord $d$, so the values of the incidence indices of the chords $a,b,c$ do not matter.

Assume now that the chord $a$ is Gaussian even and the chords $b$ and $c$ are Gaussian odd.  We assume further that the pair of chords $a$ and $b$ is linked and the pairs $a,c$ and $b,c$ are not linked (see Fig.~\ref{fig:oriented_Gaussian_parity_R3}). Then the incidence indices are  $\eps(a)=\eps(b)=-1$ and $\eps(c)=1$. We choose segments between $b$ and $c$ so that $k_e=k_o=0$. Then the number $l_o+1$ is the number of odd ends in the half of the diagram at the chord $a$ (we must count the chord $b$ besides the odd chords whose ends lie between $b$ and $c$).

\begin{figure}[h]
\centering\includegraphics[width=0.25\textwidth]{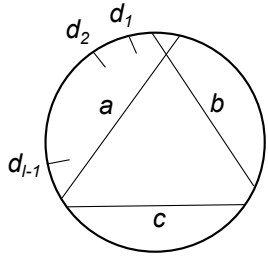}
\caption{Case $a$ and $b$ are linked}\label{fig:oriented_Gaussian_parity_R3}
\end{figure}

If $b$ and $c$ belong to the same class of crossings then $l_o$ is odd and $P^{og}_{D_1}(a)=2(l_o+1)=0\in\Z_4$. Then
$$
\eps(a)P^{og}_{D_1}(a)+\eps(b)P^{og}_{D_1}(b)+\eps(c)P^{og}_{D_1}(c)=-P^{og}_{D_1}(b)+P^{og}_{D_1}(c)=0
$$
since $b$ and $c$ have the same oriented Gaussian parity.

If $b$ and $c$ belong to the different classes then $l_o$ is even and $P^{og}_{D_1}(a)=2$. The sum
$$
\eps(a)P^{og}_{D_1}(a)+\eps(b)P^{og}_{D_1}(b)+\eps(c)P^{og}_{D_1}(c)
$$
is equal here either to $-2-1+3=0$ or $-2-3+1=-4=0\in\Z_4$.

The other linking configurations of the chords $a,b,c$ can be considered analogously.

Thus, the oriented Gaussian parity functor we defined is indeed an oriented parity functor. It remains to prove the parity functor is universal.

Let $(P,\G)$ be an oriented parity functor on free knots and $D$ be a free knot diagram. By Proposition~\ref{prop:free_parity_odd_even} the map $P_D$ is constant on each of classes of crossings $E_0, E_1, O', O''$. Hence, $P_D$ can be considered as a composition $P_D=\rho_D\circ P^{og}_D$, where $\rho_D$ maps $\G^{og}(D)$ to $\G(D)$. Note that the map $\rho_D$ is uniquely defined. From Proposition~\ref{prop:free_parity_odd_even} and definition of the group structure on $\G^{og}(D)$ it follows that $\rho_D$ is a group homomorphism. Thus, the oriented Gaussian parity functor is universal.
\end{proof}

\begin{remark}
Note that the isomorphism $\G^{og}(D)\simeq\Z_4$ in the definition of the oriented Gaussian parity functor is not canonical, so we can not consider $\G$ as a reduction of some universal oriented parity with coefficients in $\Z_4$.

For example, the parity of chords $b$ and $c$ in the diagram in fig.~\ref{fig:trivial_free_knot_diagram} should have values $1$ and $3$ in $\Z_4$ but we can not decide which chord corresponds to $1$ because they are symmetrical. If we assume the parity of the chords $b$ and $c$ were equal we would get the ordinary (nonoriented) Gaussian parity with coefficients $\Z_2$. This example shows that we need to consider parity functors if we want to keep maximal amount of information about the crossings.

\begin{figure}[h]
\centering\includegraphics[width=0.2\textwidth]{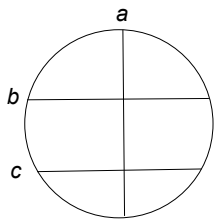}
\caption{A diagram of the free unknot}\label{fig:trivial_free_knot_diagram}
\end{figure}

On the other hand, if one considers {\em long free knot} the universal oriented parity functor does come from an oriented parity with coefficients in $\Z_4$.
\end{remark}

\begin{remark}\label{rem:oriented_gaussian_parity_invariants}
From the universal parity functor $(P^{og},\G^{og})$ we can extract an invariant of free knots:
 $L_{odd}(\mathcal K)=\left||O'|-|O''|\right|$ --- the difference between the numbers of odd chord in different classes. We have to take the absolute value of the difference because there is no canonical way to distinguish the two classes of odd chords.

The proof of invariance is trivial except for the second Reidemeister move where it follows from the property (P2) of the parity functor.

The invariant $L_{odd}$ is proportional to the invariant $L$ from~\cite{M5}. For the long free knots it can be refined with the formula $l(\mathcal K)=|O'|-|O''|$ where the class $O'$ corresponds to $1$ and $O''$ corresponds to $3$ by the identification of the coefficient groups $\G^{og}(D)$ with $\Z_4$.

For example, the free knot $\mathcal K$ in Fig.~\ref{fig:free_knot_example} has four odd chords which belong to the same class. Then $L_{odd}(K)=4$.

\begin{figure}[h]
\centering\includegraphics[width=0.2\textwidth]{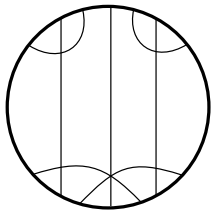}
\caption{A free knot}\label{fig:free_knot_example}
\end{figure}
\end{remark}

\section{Homotopical parity and parity functor}

The goal of this section is to describe the universal oriented parity functor for knots in a fixed surface.

Let $S$ be an oriented closed connected surface and $\mathcal K$ be a knot in $S$. Let us fix an arbitrary point $z\in S$ and consider the diagrams $D$ of the knot $\mathcal K$ such that $z\in D$ and $z$ is not a crossing of $D$. Such diagrams form a diagram category $\mathfrak K_z$. Let us define an oriented parity on the category $\mathfrak K_z$.

\begin{definition}\label{def:homotopical_parity}
    Consider the group $\Pi=\pi_1(S,z)/\langle[\mathcal K]\rangle$ where $\pi_1(S,z)$ is the fundamental group of the surface $S$ and $\langle[\mathcal K]\rangle$ is the normal subgroup generated by the homotopy class of the knot $\mathcal K$. For any diagram $D\in Ob(\mathfrak K_z)$ and any crossing $v\in\V(D)$ we define its {\em homotopical parity} $hp_D(v)\in\Pi$ to be the based left half $\hat D^l_v$ of the diagram $D$ at the crossing $v$ (Fig.~\ref{fig:based_half}).

\begin{figure}[h]
\centering\includegraphics[width=0.4\textwidth]{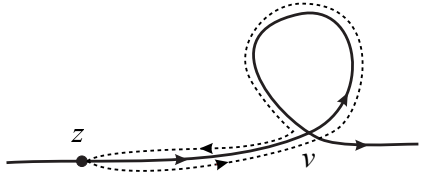}
\caption{Based left half $\hat D^l_v$ of the diagram $D$ at the crossing $v$}\label{fig:based_half}
\end{figure}

\end{definition}

\begin{theorem}
  The homotopical parity $hp$ is a well defined oriented parity on the diagram category $\mathfrak K_z$.
\end{theorem}

\begin{proof}
    Let us check the property (P0)--(P3+) of the definition of parity.

0. Let $f\colon D\to D'$ be an elementary morphism and $v\in dom(f_*)$. Then the based loops $\hat D^l_v$ and $\hat D'^l_{v}$ are homotopic except the case when $f$ is an isotopy which moves the crossing $v$ over the base point $z$ (Fig.~\ref{fig:over_basepoint}). In latter case $[\hat D^l_v]=[\alpha]$ and  $[\hat D'^l_{v}]=[\alpha\beta\alpha\beta^{-1}\alpha^{-1}]=[\mathcal K][\hat D^l_v][\mathcal K]^{-1}=[\hat D^l_v]\in\Pi$.

\begin{figure}[h]
\centering\includegraphics[width=0.6\textwidth]{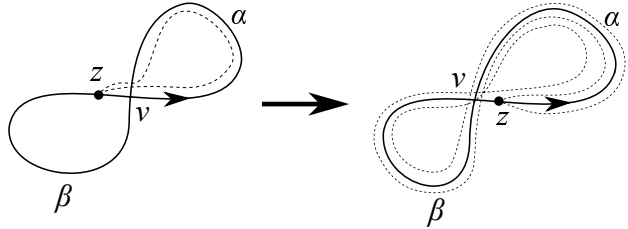}
\caption{A crossings passes over the basepoint}\label{fig:over_basepoint}
\end{figure}

1. Let $v\in\V(D)$ be a crossing which a decreasing first Reidemeister move can be applied to. Then $[D^l_v]$ is equal either to $1$ or to $[\mathcal K]=1\in\Pi$.

2. Let $u,v\in\V(D)$ be two crossings participating in a decreased second Reidemeister move (Fig.~\ref{fig:hom_parity_r2}). Then $hp_D(u)=[\gamma_1\alpha\gamma_1^{-1}]$ and $hp_D(v)=[\gamma_1\alpha\gamma_2\gamma_1\alpha^{-1}\gamma_1^{-1}]$, hence,
\begin{multline*}
  hp_D(u)hp_D(v)=[\gamma_1\alpha\gamma_1^{-1}][\gamma_1\alpha\gamma_2][\gamma_1\alpha\gamma_1^{-1}]^{-1}=\\
  [\gamma_1\alpha\gamma_1^{-1}][\mathcal K][\gamma_1\alpha\gamma_1^{-1}]^{-1}=1\in\Pi.
\end{multline*}

\begin{figure}[h]
\centering\includegraphics[width=0.25\textwidth]{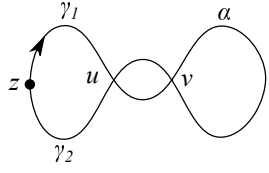}
\caption{Crossings participating in a second Reidemeister move}\label{fig:hom_parity_r2}
\end{figure}

The other configuration for second Reidemeister move can be treated analogously.

3. Let  $u,v,w\in\V(D)$ be the crossings participating in a third Reidemeister move (for example, see Fig.~\ref{fig:hom_parity_r3}). Then $hp_D(u)=[\gamma_1\alpha\gamma_1^{-1}]$, $hp_D(v)=[\gamma_1\alpha\beta\alpha^{-1}\gamma_1^{-1}]$ and $hp_D(w)=[\gamma_1\alpha\beta\gamma_1^{-1}]$. The incidence indices of the crossings are $\epsilon(u)=\epsilon(v)=-1$, $\epsilon(w)=1$. Then

\begin{multline*}
  hp_D(u)^{-1}hp_D(v)^{-1}hp_D(w)=[\gamma_1\alpha\gamma_1^{-1}]^{-1}[\gamma_1\alpha\beta\alpha^{-1}\gamma_1^{-1}]^{-1}
  [\gamma_1\alpha\beta\gamma_1^{-1}]=\\
  [\gamma_1\alpha^{-1}\gamma_1^{-1}\gamma_1\alpha\beta^{-1}\alpha^{-1}\gamma_1^{-1}\gamma_1\alpha\beta\gamma_1^{-1}]=1.
\end{multline*}

\begin{figure}[h]
\centering\includegraphics[width=0.3\textwidth]{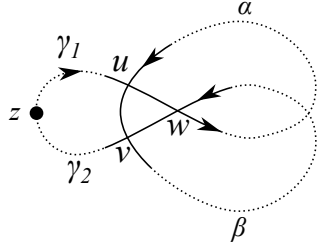}
\caption{Crossings participating in a third Reidemeister move}\label{fig:hom_parity_r3}
\end{figure}

Thus, the property (P3+) holds. Similarly, one can check the other configurations of third Reidemeister move.
\end{proof}

\begin{definition}\label{def:homotopical_parity_functor}
The {\em homotopical parity functor} on the diagram category $\mathfrak K_z$ is the canonical reduction $(HP,\tilde\Pi)$ of the homotopical parity $hp$.
\end{definition}

\begin{theorem}\label{thm:universal_homotopical_parity_functor}
  The homotopical parity $(HP,\tilde\Pi)$ is the universal oriented parity
functor on the diagram category $\mathfrak K_z$.
\end{theorem}

\begin{proof}
Let $(P,\mathcal G)$ be an oriented parity functor on $\mathfrak K_z$.

\begin{lemma}\label{lem:polygon_relation}
Let $D\in Ob(\mathfrak K_z)$ be a diagram of the knot $\mathcal K$ and $\pi$ be a cell in $S\setminus D$ with vertices $v_1,\dots, v_k\in\V(D)$ numbered counterclockwise (see Fig.~\ref{fig:polygon} left). Then $\prod_{i=1}^k P_D(v_i)^{\epsilon_\pi(v_i)}=1\in\mathcal G(D)$ where $\epsilon_\pi(v_i)$ are the incidence indices of the crossings to the cell $\pi$.
\end{lemma}

\begin{figure}[h]
\centering\includegraphics[width=0.7\textwidth]{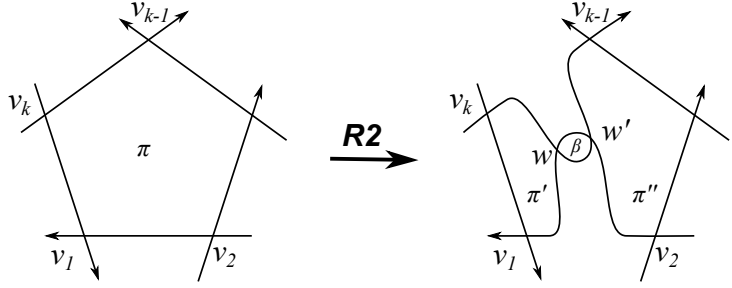}
\caption{Proof of Lemma~\ref{lem:polygon_relation}}\label{fig:polygon}
\end{figure}

\begin{proof}

We prove the statement by induction on $k$. The cases $k=1,2,3$ follow from the properties (P1),(P2) and (P3+).

Let $k>3$. Assume the statement holds for any cell with $\le k-1$ vertices. Let $\pi$ be a cell with $k$ vertices. Apply a second Reidemeister move $f\colon D\to D'$ to split $\pi$ into a bigon $\beta$, a triangle $\pi'$ and a $(k-1)$-gon $\pi''$ as shown in Fig.~\ref{fig:polygon}. Then $\epsilon_\pi(v_1)=\epsilon_{\pi'}(v_1)$, $\epsilon_\pi(v_k)=\epsilon_{\pi'}(v_k)$, $\epsilon_\pi(v_i)=\epsilon_{\pi''}(v_i)$ for $i=2,\dots,k-1$, and $\epsilon_{\pi'}(w)=\epsilon_{\beta}(w)=\epsilon_{\beta}(w')=\epsilon_{\pi''}(w')$. By the property (P2) $P_{D'}(w)P_{D'}(w')=1$, hence,
$$
P_{D'}(w)^{-\epsilon_{\beta}(w))}P_{D'}(w')^{-\epsilon_{\beta}(w')}=1.
$$
By the property (P3+)
$$P_{D'}(v_k)^{\epsilon_{\pi'}(v_k)}P_{D'}(v_1)^{\epsilon_{\pi'}(v_1)}P_{D'}(w)^{\epsilon_{\pi'}(w')}=1,$$
and by induction
$$P_{D'}(w')^{\epsilon_{\pi''}(w')}P_{D'}(v_2)^{\epsilon_{\pi''}(v_2)}\cdots P_{D'}(v_{k-1})^{\epsilon_{\pi''}(v_{k-1})}=1.$$
Thus,
\begin{multline*}
  1=P_{D'}(v_k)^{\epsilon_{\pi'}(v_k)}P_{D'}(v_1)^{\epsilon_{\pi'}(v_1)}P_{D'}(w)^{\epsilon_{\pi'}(w')}\cdot
  P_{D'}(w)^{-\epsilon_{\beta}(w))}P_{D'}(w')^{-\epsilon_{\beta}(w')}\cdot \\
  P_{D'}(w')^{\epsilon_{\pi''}(w')}P_{D'}(v_2)^{\epsilon_{\pi''}(v_2)}\cdots P_{D'}(v_{k-1})^{\epsilon_{\pi''}(v_{k-1})} =\\
  P_{D'}(v_k)^{\epsilon_{\pi}(v_k)}P_{D'}(v_1)^{\epsilon_{\pi}(v_1)}P_{D'}(v_2)^{\epsilon_{\pi}(v_2)}\cdots P_{D'}(v_{k-1})^{\epsilon_{\pi}(v_{k-1})}.
\end{multline*}
Therefore,  $\prod_{i=1}^k  P_{D'}(v_i)^{\epsilon_{\pi}(v_i)}=1$, so $\prod_{i=1}^k  P_{D}(v_i)^{\epsilon_{\pi}(v_i)}=1$.
\end{proof}

Let $D\in Ob(\mathfrak K_z)$ be a diagram of the knot $K$ such that $S\setminus D$ is a union of cells. Let $F(D)=\langle\V(D)\rangle$ be the free group generated by the set of crossings $\V(D)$. For any cell $\pi$ in $S\setminus D$  with the vertices $v_1,\dots, v_k$ (numbered counterclockwise) consider the word
$$
W_\pi=\prod_{i=1}^{k}v_i^{\epsilon_\pi(v_i)}\in F(D),
$$
and let $\mathcal H(D)=F(D)/\langle W_\pi\,|\,\pi\in \pi_0(S\setminus D)\rangle$ be the factor-group of $F(D)$ by the normal subgroup generated by the words $W_\pi$. By Lemma~\ref{lem:polygon_relation} there is a well defined homomorphism $\tilde\phi_D\colon \mathcal H(D)\to\mathcal G(D)$ given by the formula $\tilde\phi_D(v)=P_D(v)$, $v\in\V(D)$.

\begin{lemma}\label{lem:homotopical_group_presentation}
The map $\psi_D\colon v \mapsto [\hat D^l_v]$, $v\in\V(D)$ defines an isomorphism from the group $\mathcal H(D)$ to the group $\Pi=\pi_1(S,z)/\langle[\mathcal K]\rangle$.
\end{lemma}

\begin{proof}

Let $C(D)$ be the Gauss diagram of $D$. Topologically, $C(D)$ is a $1$-dimensional cell complex that consists of the core circle $S^1$ and the chords. We identify the chord set with the set $\V(D)$ of crossings of the diagram $D$. We modify the orientation of the chords: keep the orientation of the chords $v$ such that $sgn(v)=+1$ and change the orientation of the negative chords.

There is a natural projection $h\colon C(D)\to D$ which contracts any chord in $C(D)$ to a point. The map $h$ is a homotopy equivalence.

For any crossing $v\in\V(D)$ the preimage of the based left half of the diagram at the crossing is a loop in the Gauss diagram that goes along the chord $v$ in the direction of the orientation we defined (Fig.~\ref{fig:gauss_knot_half}).

\begin{figure}[h]
\centering\includegraphics[width=0.5\textwidth]{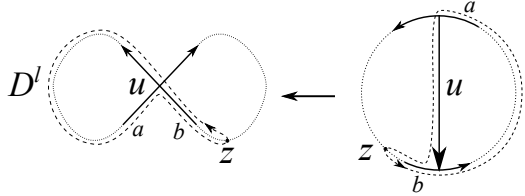}
\caption{Based left half on the Gauss diagram}\label{fig:gauss_knot_half}
\end{figure}

Let $\tilde C=C(D)\cup_{S^1}D^2$ be the cell complex obtained from the Gauss diagram by gluing a disk $D^2$ along the core circle of the diagram. Let $g$ be the map on $\tilde C$ which contracts the disk to a point. The image of $g$ is a bouquet of loops $\bigvee_{v\in\V(C)} S^1_v$, and $g$ is a homotopy equivalence. The map $h$ extends to a homotopy equivalence $h\colon \tilde C\to D\cup_{\mathcal K}D^2$.
Thus, we have an isomorphism of fundamental groups $\pi_1(D\cup_{\mathcal K}D^2)\simeq\pi_1(\tilde C)\simeq \pi_1(\bigvee_{v\in\V(C)} S^1_v)=F(D)$. This isomorphism $g_*(h_*)^{-1}$ identifies the based left half $\hat D^l_v$ of a crossing $v\in\V(D)$ with the generator $v\in F(D)$.

Consider the cell complex $\tilde S=S\cup_{\mathcal K}D^2$ which obtained by the gluing a disk along the knot $\mathcal K$ to the surface $S$. Then $\pi_1(\tilde S)=\pi_1(S)/\langle[\mathcal K]\rangle=\Pi$. On the other hand, the complex $\tilde S$ can be obtained from $D\cup_{\mathcal K}D^2$ by gluing the cells $\pi\in \pi_0(S\setminus D)$. Hence, $\pi_1(\tilde S)=\pi_1(D\cup_{\mathcal K}D^2)/\langle[\partial\pi]\,|\,\pi\in \pi_0(S\setminus D)\rangle$ where $\partial\pi$ is the based loop for the boundary of the cell $\pi$.

\begin{figure}[h]
\centering\includegraphics[width=0.7\textwidth]{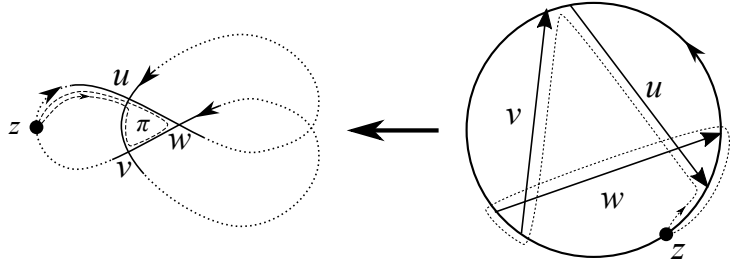}
\caption{Cell loop $\pi$ corresponds to the word $u^{-1}v^{-1}w=W_\pi$}\label{fig:gauss_polygon}
\end{figure}

The isomorphism $g_*(h_*)^{-1}$ maps the homotopy class $[\partial\pi]$ to the word $W_\pi\in F(D)$ (see for example Fig.~\ref{fig:gauss_polygon}). Thus,
$$\Pi=\pi_1(\tilde S)\simeq F(D)/\langle W_\pi\,|\,\pi\in \pi_0(S\setminus D)\rangle=\mathcal H.$$
\end{proof}

Thus, when the complement to a diagram $D$ in the surface $S$ consists of cells, we have a well defined homomorphism $\phi_D\colon\tilde\Pi(D)=\Pi\stackrel{\tilde\phi_D\circ\psi^{-1}_D}{\longrightarrow}\mathcal G(D)$, where for any $v\in\V(D)$  $$\phi_D(HP_D(v))=\phi_D([\hat D^l_v])=\tilde\phi_D(v)=P_D(v).$$ The equality $\tilde\Pi(D)=\Pi$ follows from the fact the elements $[\hat D^l_v]$, $v\in\V(D)$, generate $\Pi$.

Assume now that $D$ is a diagram such that $S\setminus D$ contains components with handles. Then one can apply a sequence of increasing second Reidemeister moves $f\colon D\to D'$ and get a diagram $D'$ such that $S\setminus D'$ splits into cells. The map $f_*\colon \V(D)\to \V(D')$ is an embedding, and $\tilde\Pi(D)$ can be identified with the subgroup in $\tilde\Pi(D')$ generated by the elements $HP_{D'}(f_*(v)),v\in\V(D)$.

The image $im(\mathcal G(f))\subset \mathcal G(D')$ contains elements $P_{D'}(f_*(v))=\phi_{D'}(HP_{D'}(f_*(v)))$, $v\in\V(D)$. Hence, $\phi_{D'}(\tilde\Pi(D))\subset im(\mathcal G(f))$. Since $\mathcal G(f)$ is invertible on $im(\mathcal G(f))$, we can define a homomorphism $\phi_D\colon \tilde\Pi(D)\to \mathcal G(D)$ as the composition $\mathcal G(f)^{-1}\circ\phi_{D'}$. By definition, for any $v\in\V(D)$ we have
$$\phi_D(HP_D(v))=\mathcal G(f)^{-1}(P_{D'}(f_*(v)))=P_D(v).$$

Thus, for any parity functor $(P,\mathcal G)$ we have constructed a homomorphism from the homotopical parity functor to $(P,\mathcal G)$. Since the homotopical functor is reduced, the homomorphism is unique. Theorem~\ref{thm:universal_homotopical_parity_functor} is proved.
\end{proof}

\begin{remark}
Let $z,z'\in S$ be two points and $D$ be a diagram of the knot $\mathcal K$ which contains $z$ and $z'$ as points but not as crossings. Then $D\in Ob(\mathfrak K_z)\cap Ob(\mathfrak K_{z'})$. The coefficient groups $\Pi_z=\pi_1(S,z)/\langle[\mathcal K]\rangle$ and $\Pi_{z'}=\pi_1(S,z')/\langle[\mathcal K]\rangle$ are isomorphic (by the conjugation with the arc $zz'$ of $D$), and the homotopical parity values $hp_D(v)$, $v\in\V(D)$, are identified by this isomorphism. This means that the kernels of the homomorphisms from $F(D)=\langle\V(D)\rangle$ to $\Pi_z$ and $\Pi_{z'}$ coincide and the canonical reductions $(\tilde\Pi_z)_D$ and $(\tilde\Pi_{z'})_D$ are the same factor-group of  $F(D)$. Hence, the canonical reduction of the homotopical parity does not depend on the choice of the base point.

Thus, the homotopical parity functor $(HP,\tilde\Pi)$ can be uniquely extended to the category  $\mathfrak K$ of all diagrams of the knot $\mathcal K$ (without fixing a basepoint). This oriented parity functor is universal.
\end{remark}

\begin{corollary}
The homotopical parity $hp$ is universal oriented parity on the diagram category $\mathfrak K_z$.
\end{corollary}
\begin{proof}
The statement follows from Theorem~\ref{thm:universal_homotopical_parity_functor}, Propostions~\ref{prop:universal_associated_parity},\ref{prop:associated_reduction_parity}, and the fact the homotopy classes of based halves $\hat D^l_v$ of crossings $v$ of knot diagrams $D$ generate the fundamental group $\pi_1(S,z)$.
\end{proof}

Let $H_1(S,\Z)$ be the $1$-dimensional integral homology group and $H=H_1(S,\Z)/\langle[\mathcal K]\rangle$ be the factor-group by the homology class of the knot $\mathcal K$. There is a natural homomorphism $\rho$ from $\Pi=\pi_1(S,z)/\langle[\mathcal K]\rangle$ to $H$. The compositions $\overline{hp}_D=\rho\circ hp_D\colon\V(D)\to H$, $D\in Ob(\mathfrak K_z)$, define an oriented parity with coefficients in the group $H$, called the {\em homological parity} on diagrams of the knot $\mathcal K$ in the surface $S$.

\begin{corollary}
The homological parity $\overline{hp}$ is the universal parity among all oriented parities with coefficients in abelian groups on the diagram category $\mathfrak K_z$. In other word, for any oriented parity $p$ with coefficients in an abelian group $A$ there is a unique homomorphism $\psi\colon H\to A$ such that $p_D=\psi\circ \overline{hp}_D$ for all $D\in Ob(\mathfrak K_z)$.
\end{corollary}
\begin{proof}
By universality of the homotopical parity there is a unique homomorphism $\tilde\psi\colon \Pi\to A$ such that $p_D=\tilde\psi\circ hp_D$ for all $D$. Since $H=\Pi/[\Pi,\Pi]$ is the abelinization of the group $\Pi$ and $A$ is abelian, there is a unique homomorphism $\psi$ such that $\tilde \psi=\psi\circ\rho$. Then $p_D=\psi\circ\rho\circ hp_D=\psi\circ \overline{hp}_D$.
\end{proof}

Note that the homology does not depend on a base point. Therefore, the homological parity $\overline{hp}$ can be considered as an oriented parity on the category $\mathfrak K$ of all diagrams of $\mathcal K$ in $S$. The parity is universal among all oriented parities on $\mathfrak K$ with coefficients in an abelian group.

If $\mathcal K$ is a classical knot, i.e. knot on the sphere $S^2$ we get the following statement (cf.~\cite[Corollary 3.3]{IMN1}). It shows that generalization from parities to oriented parity functors is unproductive for classical knots.

\begin{corollary}
Any oriented parity functor on diagrams of a classical knot is trivial.
\end{corollary}
\begin{proof}
Let $(P,\mathcal G)$ be an oriented parity functor. By universality, there is a family of homomorphisms $\psi_D\colon\tilde\Pi(D)\to\mathcal G(D)$, $D\in Ob(\mathfrak K_z)$ from the homotopic parity $(HP,\tilde\Pi)$. Since the group $\tilde\Pi(D)$ is identified with a subgroup in the trivial group $\pi_1(S^2,z)/\langle[\mathcal K]\rangle=\{1\}$, for any diagram $D$ and any crossing $v\in\V(D)$  $HP_D(v)=1$ and $P_D(v)=\psi_D(HP_D(v))=1$.
\end{proof}

\section{Linking number invariant}

Recall that for an oriented two-component link $L=K_1\cup K_2$ its linking number is the sum
$$
lk(L)=\frac 12\sum_{c\in K_1\cap K_2}sgn(c)
$$
of the signs of the crossings where two different components intersect. Using the link parity $lp$ from Example~\ref{exa:link_parity} we can rewrite the definition as follows
$$
lk(L)=\frac 12\sum_{c\in\V(L)\colon lp_L(c)=1}sgn(c).
$$
The last formula admits a broad generalization. If there is a parity on the diagram category, one can assign a linking number to each nonzero parity value.

\begin{definition}
Let $\mathcal K$ be a virtual knot or a knot in a surface, $\mathfrak K$ be its diagram category, and $p$ be an oriented parity on $\mathfrak K$ with coefficients in a group $G$. Given a diagram $D$ of $\mathcal K$, we define the {\em linking invariant} as the element
$$
lk^p(\mathcal K)=\sum_{v\in\V(D)}sgn(v)\cdot p_D(v)^{sgn(v)}-w(D)\cdot 1\in\Z[G]
$$
in the group algebra of $G$. Here $w(D)=\sum_{v\in\V(D)}sgn(v)$ is the writhe of the diagram $D$.

For any $g\in G$, $g\ne 1$, the coefficient $lk^p_g(\mathcal K)=\sum_{v\in\V(D)\colon p_D(v)^{sgn(v)}=g}sgn(v)$ is called the {\em linking number} of the knot $\mathcal K$ at the parity value $g$.
\end{definition}
The independence of $lk^p(\mathcal K)$ on the choice of the diagram $D$ follows directly from the properties (P0),(P1),(P2).

In fact, the linking invariant can be considered as an example of virtual knot invariants called {\em index polynomials}~\cite{Cheng, IKL, K2}.
An index polynomial can be defined whenever one has a chord index~\cite{Cheng2} on the diagrams. For any oriented parity $p$ the value $p_D(v)^{sgn(v)}$, $v\in\V(D), D\in\mathfrak K$, is an example of such a chord index.

Let $(P,\mathcal G)$ be an oriented parity functor on the diagram category $\mathfrak K$. For any diagram $D$ the linking invariant $lk^P(D)$ is an element of the algebra $\Z[\mathcal G(D)]$. In general case, there is no canonical way to identify the groups $\mathcal G(D)$. We have to weaken the linking invariant to get rid of dependence on coefficients groups.

\begin{definition}
Let $\mathcal K$ be a virtual knot or a knot in a surface, $\mathfrak K$ be its diagram category, and $(P,\mathcal G)$ be an oriented parity functor on $\mathfrak K$. For any diagram $D$ of $\mathcal K$, define the {\em linking multiset} by the formula
$$LS^P(D)=\{lk^P_g(D)\,|\, g\in \mathcal G(D), g\ne 1 \}.$$ One can consider the linking multiset as an element of $\Z[\Z]$:
$$
LS^P(D)=\sum_{k\in\Z}\left|\{g\in\mathcal G(D)\setminus\{1\}\,|\, lk^P_g(D)=k\}\right|\cdot\boldsymbol k.
$$
\end{definition}

The invariance of the linking multiset $LS^P(\mathcal K)$ follows from the invariance of the linking invariant.

\begin{example}
Let $\mathcal K$ be the virtual trefoil considered as a knot in the torus (Fig.~\ref{fig:virtual_trefoil2}). Then the homology class of $\mathcal K$ is $2\lambda+\mu$ where $\lambda$ is the longitude and $\mu$ is the meridian of the torus. Hence, the coefficient group of the homological parity is
$$\mathcal H=H_1(T^2,\Z)/\langle[\mathcal K]\rangle=\Z\langle\lambda,\mu\rangle/\langle 2\lambda+\mu\rangle\simeq\Z\lambda.$$
The homological parity of the crossings is $hp(u)=\lambda$, $hp(v)=\lambda+\mu=-\lambda$. The linking invariant is equal to $lk^{hp}(\mathcal K)=(-1)\cdot [(-1)\lambda]+1\cdot[1\cdot(-\lambda)]=2\cdot(-\lambda)$. The linking multiset is $LK^{hp}(\mathcal K)=\{2\}=1\cdot\mathbf 2$.

\begin{figure}[h]
\centering\includegraphics[width=0.3\textwidth]{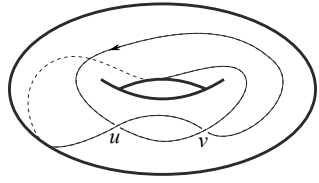}
\caption{Virtual trefoil}\label{fig:virtual_trefoil2}
\end{figure}
\end{example}

Let $\mathcal K$ be a flat or free knot or an immersed curve in a fixed surface. In this case we can not use signs of crossings and lose about a half information on linking numbers.

\begin{definition}
Let $\mathcal K$ be a flat or free knot or an immersed curve in a fixed surface, $\mathfrak K$ be its diagram category, and $p$ be an oriented parity on $\mathfrak K$ with coefficients in a group $G$. Let $G_{inv}=\{g\in G\,|\, g^2=1,g\ne 1\}$ be the set of involutions in $G$, and
$G_{ni}=G\setminus (G_{inv}\cup\{1\})$ be the set of elements of order greater then $2$. Let $D$ be a diagram of the knot $\mathcal K$. Define the {\em linking invariant} of the knot $\mathcal K$ as the pair $lk^p(\mathcal K)=(lk^p_{inv}(\mathcal K), lk^p_{ni}(\mathcal K))$ where
$$
lk^p_{inv}(\mathcal K)=\sum_{g\in G_{inv}}|\{v\in\V(D)\,|,\ p_D(v)=g\}|\cdot g\in \Z_2[G_{inv}],
$$
\begin{multline*}
lk^p_{ni}(\mathcal K)=\\
\sum_{g\in G_{ni}}\left(|\{v\in\V(D)\,|,\ p_D(v)=g\}|-|\{v\in\V(D)\,|,\ p_D(v)=g^{-1}\}|\right)\cdot g\in \Z[G_{ni}].
\end{multline*}

The element $lk^p_{inv}(\mathcal K)$ is called the {\em involutive part of the linking invariant}, and $lk^p_{ni}(\mathcal K)$ is the {\em non involutive part}.

For an element $g\in G$, $g\ne 1$, the linking number $lk^p_g(\mathcal K)$ of the knot $\mathcal K$ at the parity value $g$ is the coefficient of the linking invariant at $g$, i.e. $lk^p_g(\mathcal K) = |p_D^{-1}(g)|\in\Z_2$ for $g\in G_{inv}$, and
$lk^p_g(\mathcal K) =|p_D^{-1}(g)|-|p_D^{-1}(g^{-1})|$ for $g\in G_{ni}$.

Note that $lk^p_{g^{-1}}(\mathcal K)=-lk^p_{g}(\mathcal K)$ for any $g\ne 1$.
\end{definition}

\begin{proposition}
The linking invariant $lk^p$ and the linking numbers $lk^p_g$, $g\in G\setminus\{1\}$, do not depend on the choice of the diagram $D$.
\end{proposition}

\begin{proof}
The only nontrivial part is invariance under the second Reidemeister move. Let a diagram $D'$ be obtained from $D$ by an increasing second Reidemeister move, and $v_1,v_2$ be the new two crossings in $D'$. Let $p_{D'}(v_1)=g$, $g\in G$. Then $p_{D'}(v_2)=g^{-1}$. By the property (P0) $lk^p_{g'}(D')=lk^p_{g'}(D)$ for all $g'\ne g,g^{-1}$. If $g=g^{-1}\in G_{inv}$ then $lk^p_{g}(D')=lk^p_{g}(D)+2=lk^p_{g}(D)$. If $g\in G_{ni}$ then $lk^p_{g}(D')=lk^p_{g}(D)+1-1=lk^p_{g}(D)$, analogously, $lk^p_{g^{-1}}(D')=lk^p_{g^{-1}}(D)$. Thus, the linking numbers are invariant.
Then $lk^p(D)=lk^p(D')$.
\end{proof}

Denote $\bar G_{ni}=G_{ni}/\sigma$ where $\sigma(g)=g^{-1}$, $g\in G_{ni}$. Then for any $x\in \bar G_{ni}$, $x=\{g,g^{-1}\}$, the value $lk^p_x(D)=|lk^p_g(D)|=|lk^p_{g^{-1}}(D)|$ is well defined.

\begin{definition}
Let $\mathcal K$ be a flat or free knot or an immersed curve in a fixed surface, $\mathfrak K$ be its diagram category, and $(P,\mathcal G)$ be an oriented parity functor on $\mathfrak K$. For any diagram $D$ of $\mathcal K$, the {\em involutive linking multiset} is $LS^P_{inv}(D)=\{lk^P_g(D)\,|\, g\in \mathcal G(D)_{inv}\}$, and the {\em non-involutive linking multiset} is $LS^P_{ni}(D)=\{lk^P_x(D)\,|\, x\in \overline{\mathcal G(D)}_{ni}\}$. In other words, one can consider the linking multiset as an element of :
\begin{gather*}
LS^P_{inv}(D)=\sum_{k\in\Z_2}\left|\{g\in\mathcal G(D)_{inv}\,|\, lk^P_g(D)=k\}\right|\cdot{\boldsymbol k}\in\Z[\Z_2],\\
LS^P_{ni}(D)=\sum_{k\in\Z, k\ge 0}\left|\{x\in\overline{\mathcal G(D)}_{ni}\,|\, lk^P_x(D)=k\}\right|\cdot{\boldsymbol k}\in\Z[\N\cup\{0\}].
\end{gather*}
\end{definition}

\begin{example}
Let $(P^{og},\mathcal G^{og})$ be the oriented Gaussian parity on free knots. Then the noninvolutive linking multiset is $LS^P_{ni}(\mathcal K)=\{ L_{odd}(\mathcal K)\}$ (see Remark~\ref{rem:oriented_gaussian_parity_invariants}). The involutive linking multiset is trivial $LS^P_{inv}(\mathcal K)=\{|E_1|\}=\{0\}$ because the class $E_1$ contains even number of crossings.
\end{example}

\bibliographystyle{amsplain}

\end{document}